\documentclass[reqno,10pt,a4paper]{amsart}
\RequirePackage{tikz-cd}
  \usepackage{amsmath}
\usepackage{amssymb}
\usepackage{latexsym}
\usepackage{epsf}
\usepackage{youngtab}
\usepackage{graphics}
\usepackage{tikz}
 \usepackage{comment}
\usetikzlibrary{decorations.pathreplacing}
\usepackage[dvips]{rotating}

\usepackage[all]{xy}
\usepackage{multicol}
\usepackage{hyperref} 
\usepackage[normalem]{ulem}  

\usepackage{bbm} 

\UseComputerModernTips
\CompileMatrices

\DeclareMathOperator{\Hom}{Hom}

\DeclareMathOperator{\End}{End}
\renewcommand{\ge}{\geqslant}

\newcommand{\U}{\mathcal{U}}

\newcommand{\C}{\mathbb{C}}
\newcommand{\QQ}{\mathbb{Q}}  
\newcommand{\N}{\mathbb{N}}

\newcommand{\Z}{\mathbb{Z}}

\DeclareMathOperator{\Ind}{Ind}
\DeclareMathOperator{\Res}{Res}

\newcommand{\rb}[2]{\raisebox{#1}{#2}}
\newcommand{\ig}{\includegraphics}

\newcommand{\beq}{\begin{equation}}
\newcommand{\eq}{\end{equation}} 

\newcommand{\blue}[1]{\textcolor{blue}{#1}}
\newcommand{\bluex}[1]{}
\newcommand{\red}[1]{\textcolor{red}{#1}}
\newcommand{\redx}[1]{}
\newcommand{\ppm}[1]{\red{#1}}
\newcommand{\ppmx}[1]{} 
\newcommand{\ppmy}[1]{}   
\newcommand{\ppmm}[1]{{#1}}
\newcommand{\ca}[1]{\blue{#1}}
\newcommand{\caa}[1]{{#1}}
\newcommand{\kn}{k}
\newcommand{\nk}{n}  
\newcommand{\dl}{d}  
\newcommand{\dell}{\dl}  
\newcommand{\PP}{P_\nk^\dl(\delta)} 
\newcommand{\spine}{\mathcal{S}_n(\emptyset,\emptyset)}
\newcommand{\Sym}{\mathfrak{S}}
\newcommand{\cell}{\mathcal{S}}
\newcommand{\SSS}{\mathcal{S}} 
\newcommand{\T}{T^2}  
\newcommand{\ignore}[1]{}  

\newcommand{\rmP}{{\mathtt P}} 
\newcommand{\rmPe}{{\rmP}^{(e)}} 

\newcommand{\Part}{\rmP}  

\newcommand{\confer}[1]{$\;$(Cf. (#1).)}  

\newcommand{\Gram}{\Gamma}

\begin{document}
\theoremstyle{plain}
\newtheorem{thm}{Theorem}[section]
\newtheorem{prop}[thm]{Proposition}
\newtheorem{cor}[thm]{Corollary}
\newtheorem{clm}[thm]{Claim}
\newtheorem{lem}[thm]{Lemma}
\newtheorem{conj}[thm]{Conjecture}
\theoremstyle{definition}
\newtheorem{defn}[thm]{Definition}
\newtheorem{rem}[thm]{Remark}
\newtheorem{eg}[thm]{Example}
\newtheorem*{rem*}{Remark}



\newcounter{minidef}[section]
\renewcommand{\theminidef}{\thesection.\arabic{minidef}}
\newcommand{\mdef}{\refstepcounter{thm} 
\medskip \noindent ({\bf \thethm}) }

\newcounter{minicapt}
\newenvironment{capt}[1]
{\refstepcounter{minicapt}\stepcounter{figure}
\begin{center}\sf Figure \theminicapt. }{ \end{center}}

\newcommand{\soutx}[1]{}  


\title%
{Semisimplicity criterion for 
$2$-tonal Partition algebras}
\author{C. A.  Ahmed }
\author{G.  M. Benkart$^\dagger$}
\thanks{$\dagger$ Sadly our friend and lead-collaborator passed on before this write-up of our results was complete.} 
\author{O. H. King \and  P. P. Martin
\and A. E. Parker} 
\address{Department of Mathematics \\ University of Wisconsin-Madison
  \\ Madison
   \\ USA}
\address{Department of Mathematics \\ University of Leeds \\ Leeds,
  LS2 9JT \\ UK}
  \address{Department of Mathematics \\ College of Science\\ University of Sulaimani\\ Sulaymaniyah,
  Kurdistan Region \\ Iraq}
\email{chwas.ahmed@univsul.edu.iq}
\email{oking@dhfs.uk}  
\email{p.p.martin@leeds.ac.uk}
\email{A.E.Parker@leeds.ac.uk}

\newcommand{\heredity}{Brauer contravariant}

\begin{abstract}
We determine the semisimplicity criterion for even partition algebras over the complex field. 
Specifically we 
prove that the 
even/2-tonal    
partition algebras $P^2_n(\delta)$ over $\C$ are semisimple for all $n$ if and only if parameter $\delta \not\in \N_0$.

\soutx{ 
In passing we take the opportunity to develop more formally than usual 
(cf. e.g. 
\mbox{\cite{JamesMurphy79,MartinSaleur94b}}) 
the toolkit of 
`\heredity' 
forms (their gram matrices, Smith forms and so on). 
}
\end{abstract}

\maketitle






\newcommand{\tab}[1]{#1}
\newcommand{\taboo}[1]{#1}

\vspace{1cm}

\tableofcontents 

\vspace{1cm}

\section{Outline}  

\taboo{Partition algebras and categories are central to several areas of modern mathematics. Originally coming from computational statistical mechanics \cite{Martin91,MartinSaleur93,Martin94}, they have been studied from the perspective of    Schur--Weyl duality (see for example \cite{Jones94,benkart2019partition}),    
combinatorics 
\cite{Halverson05,orellana2007partition};  
representation theory (including Kazhdan--Lusztig theory \cite{MartinWoodcock98},  
geometric representation theory, symmetric group representation theory \cite{MartinWoodcock98,Deligne07,KhovanovSazdanovic20});   
topological quantum field theory \cite{AlvarezMartin07,Comes20,KhovanovSazdanovic20};  
probability \cite{VershikNikitin06}; 
geometric complexity \cite{bowman2013partitionalgebrakroneckercoefficients}  and many other areas. } 
The tonal partition algebras are  subalgebras of the partition algebras 
\taboo{also} 
with many intriguing properties.
They 
\taboo{too}
have been studied from the perspective of Schur--Weyl duality (see for example 
\cite{
kosuda5party,
Kosuda08,
orellana2007partition,tanabe1997centralizer}, and cf. e.g. \cite{Bloss03});
combinatorics (e.g. \cite{amm1,orellana2007partition}) 
and representation theory (e.g. \cite{amm2019tonal,Kosuda6Irr}), but many features around 
modular and geometric 
representation theory, 
statistical mechanics
and topological quantum field theory 
(cf. e.g. \cite{
AlvarezMartin07} 
and references therein%
\footnote{Since the work reported here was completed, the useful literature on 
ordinary partition categories and TQFT 
has grown further. In this and other cases where updated references might help the reader, we may add them 
directly, or 
in footnote. Here for example we can add \cite{KhovanovSazdanovic20}.})
remain relatively unexplored. 
Their connections to statistical mechanics are significantly less straightforward than the partition algebra itself (motivations around block-spin renormalisation for the Potts model, for example, cf. \cite{
banjo13,Bloss03,Martin91,MartinElgamal04
}). 
However they do inherit from the partition algebra the `thermodynamic limit' properties 
(in particular of the representation category, as for example in \cite{Martin94}) 
that make many aspects of representation theory tractable, not to mention beautiful. 
Historically one of the first questions to be 
resolved 
for the partition algebra 
$P_n(\delta)$ 
over $\C$ 
was semisimplicity criteria with respect to the parameter $\delta$. 
The proof 
(in  \cite{MartinSaleur94b})
relies on the tractability of a nice combinatorial problem in the Bell number/Stirling number realm. 
A strategy for addressing the problem for tonal partition algebras follows the same line, thus meeting a commensurately harder combinatorial problem.
In particular this is the situation for the simplest new case --- the 2-tonal 
(or even) case
--- the restriction to even partitions.
Fortunately both this combinatorial problem has been solved --- by Benkart and Moon \cite{benkart2017walks}; and the `luck' needed for the bound this provides to be saturated, holds out. In this paper we give the proof.

\medskip 

The restriction of non-semisimplicity to integral $\delta$ cases is 
significant for a number of reasons. For example in the ordinary partition algebra case this can be seen as 
a 
sign of the relevance of alcove-geometric methods and Kazhdan--Lusztig theory 
\cite{MartinWoodcock98}.
When these apply, 
the geometrical features tend to rigidify the possibilities for non-semisimple $\delta$ - in general not necessarily to integers in a given parameterisation, but for example to roots of unity, as classically for Lie Theory and quantum groups \cite{AndersenJantzenSoergel94,Soergel97a}. 
This 
can  
be compared for example with the Kadar--Yu algebras \cite{KadarMartinYu12} -- superficially similar constructions, but where it is known that the range of non-semisimple cases is 
significantly  
wilder, and where there is correspondingly no known role for Kazhdan--Lusztig theory in general. 

\medskip 

We follow 
the method of 
\cite{MartinSaleur94b} 
where the semisimplicity criterion is given
for the partition algebras,
lifting 
lemmas where necessary to $\dell$-tonal partition algebras  
$\PP \;$
(either for all $d$ or for $d=2$).
We thus obtain the criterion 
for tonality case $d=2$, formulated 
as in Theorem~\ref{th:main}
below.

\medskip 

We consider the $\dl$-tonal partition algebras
over the complex field, 
organised as in the statistical mechanical context of the partition algebra, thus  
firstly fixing the tonality $\dl$ and 
parameter $\delta$, 
and then considering the collection of these algebras over all ranks $\nk$
(corresponding to taking the thermodynamic limit of a fixed physical model \cite{Martin91}). 
It is 
relatively 
easy to show 
(for example using the Potts functor \cite{Martin91,Martin08a}, as  reviewed in \S\ref{ss:PottsFunctors})
that this 
fixed $\delta$ collection is eventually 
non-semisimple,
i.e. for some rank and then every larger rank,
if $\delta \in \N$,
and for every positive rank of form $\dl  n$  if $\delta=0$
(see below for a proof and references).
Here we answer, 
in the negative, 
the 
harder question of whether there are any 
other non-semisimple cases  
for tonality $\dl =2$.
Apart from leaning on extensive representation theoretic machinery, 
the proof leans heavily 
for inspiration 
on 
the 
beautiful algebraic-combinatorial 
results of Benkhart and Moon \cite{benkart2012planar,benkart2017walks},
which 
appropriately 
generalise a key result needed for the proof of the 
ordinary partition algebra case in \cite{MartinSaleur94b}.

\medskip 


\begin{figure} 
\[
\includegraphics[width=12.8cm]{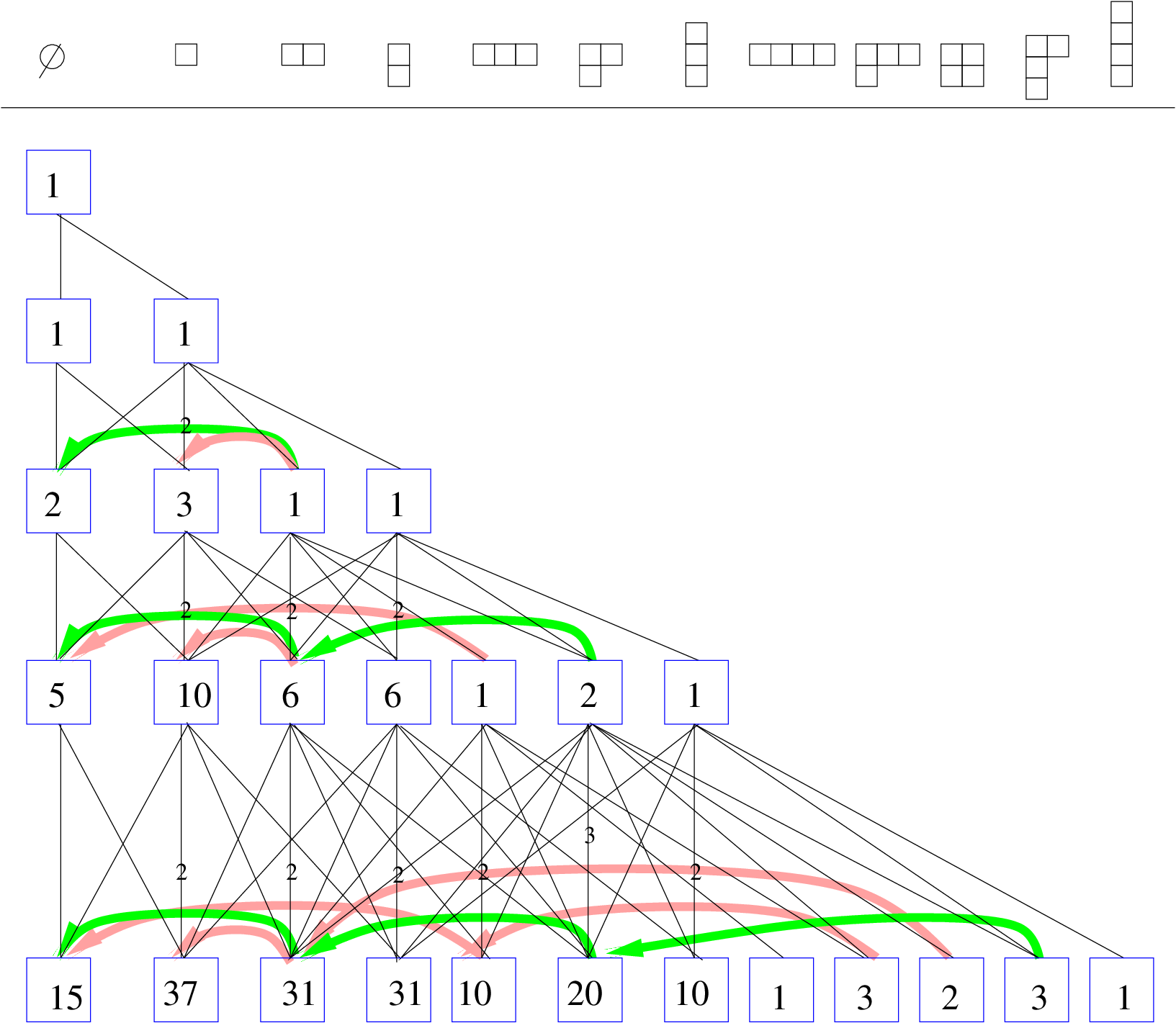}
\]
\caption{Augmented Bratelli diagram for 
algebras 
$P_0 \subset P_1 \subset ...\subset P_4$. 
Vertices are standard modules with index as shown at the top of their column; and dimension shown in the box. Black edges indicate restriction rules, with multiplicities, so dimensions can be checked.
\textcolor{green}{Green} arrows indicate module morphisms for $\delta=1 \in \C$
(see main text for commentary).
\textcolor{pink}{Pink} arrows indicate  morphisms for $\delta=2$.
(Morphisms for other $\delta$s 
omitted 
to avoid clutter.)
\label{fig:PnBratelli}}
\end{figure}

\begin{figure}[!htbp]
\[
\includegraphics[width=5.5in]{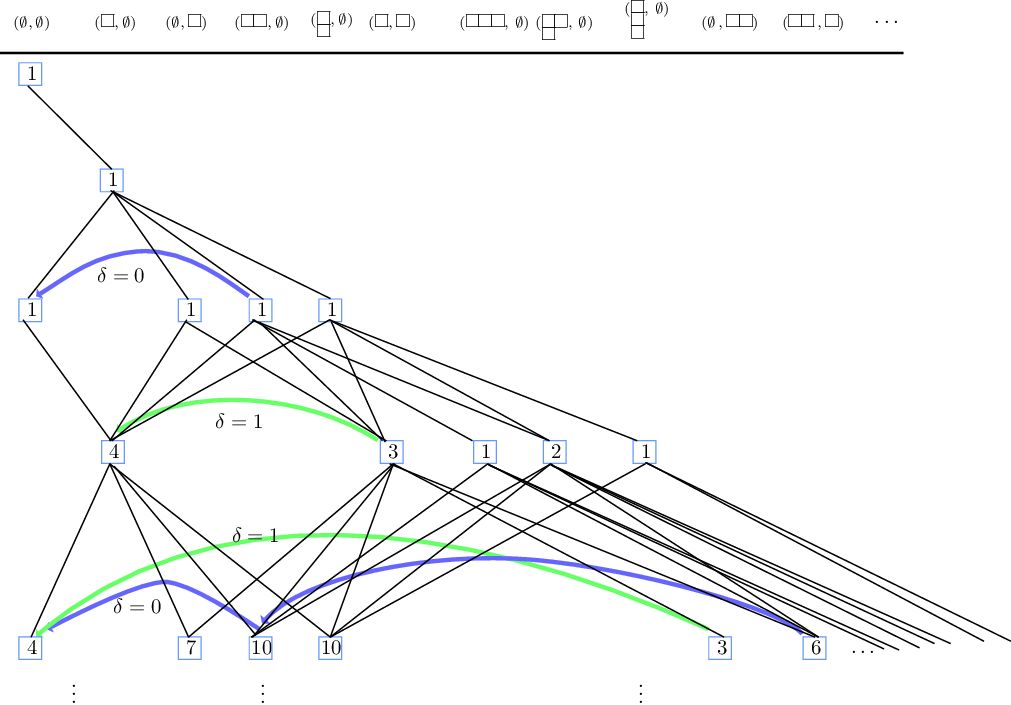}
\]
\caption{
Augmented 
Bratelli diagram for $P^2_n(\delta)$ up to $n=4$ (cf. Fig.\ref{fig:PnBratelli}), 
but here truncated 
to exclude some modules 
at rank 4, 
such as for $\lambda = \left( (1^2),(1)\right)$.
The 
coloured 
arrows represent the maps  
between the standard modules for the specified $\delta$. 
For example, 
on level $\nk=4$ when $\delta=1$ there is a map from 
$\SSS_4(\emptyset, (2))$ to 
$\SSS_{4}(\emptyset,\emptyset)$, indicated by a green arrow, with image of dimension 3.
$\;$ 
(See (\ref{pa:Trep}).)
\label{fig:bratvtwo}}
\end{figure}

To set the scene, 
we now briefly outline the problem 
(explicit  
details are
in \S\ref{ss:defns}). 
First recall that 
over $\C$ 
the partition algebras $P_n$ are generically semisimple, 
and in all cases with $\delta\neq 0$ the standard modules $\SSS_n(\lambda)$ give a basis for 
the Grothendieck group,
so that 
much of their combinatorial representation theory  
is encoded in the standard restriction diagram 
(generically the Bratelli diagram, or else 
a suitably enhanced version thereof)
for the tower of algebras 
$P_n \subset P_{n+1}$
in the generic case.
See Figure~\ref{fig:PnBratelli}.

In Fig.\ref{fig:PnBratelli}  
we also record the morphisms between standard modules 
in the cases $\delta = 1,2$ 
(generically there are no such morphisms, by Schur's Lemma).
The Figure encodes a lot of information. 
However here we
will only need some specific observations from it.


One can show, 
for example
using 
suitable 
gram matrices, 
that 
$\SSS_n(\emptyset)$ is simple for all $n$ for all but a given set of 
$\delta$ values (see \cite{MartinSaleur94b} - we review this in 
\S\ref{ss:gram-tower-semi}). 
It is this step that 
uses  
the combinatorial equalities. 
One can then show using Frobenius reciprocity and the connectedness of Fig.\ref{fig:PnBratelli} 
that $P_n$ is semisimple for every $n$ for all but this set of $\delta$ values 
(again see \cite{MartinSaleur94b}, and here (\ref{pa:indicator})). 

The corresponding Bratelli diagram for $P^2_n$ is drawn in Figure~\ref{fig:bratvtwo}. 
The details of the arguments for the two main steps are both different here. 
However  
by comparing  
Figures \ref{fig:PnBratelli} and \ref{fig:bratvtwo} 
one can already 
anticipate  
that the setups are at least  
structurally  
similar; 
and
in this paper,
in \S\ref{ss:maintheo} and \S\ref{ss:2tonal}, 
we show that both the gram matrix and Frobenius reciprocity strategies lift 
(complete with new combinatorial identities)
to prove the new 
semisimplicity criteria.   


\medskip \vspace{.1in} 

The key mathematics for this paper was completed during a joyful  and inspiring visit of Georgia to the University of Leeds.
Due to the intervention of life --- in the form of multiple happy births and sad deaths --- the write-up 
has taken much longer to complete.

\medskip 

\noindent
{\bf Acknowledgements}.
We would like to thank Tom Halversen for his 
guidance and 
advice  
that Georgia would be happy to be kept  as an author, under the sad circumstances.
PPM, OHK and AEP thank EPSRC for support from
the grant
EP/L001152/1.
PPM and AEP also thank EPSRC for support from the programme grant 
EP/W007509/1.


\section{Definitions 
and basic properties
of 
tonal partition 
algebras $\PP$}    \label{ss:defns}

\newcommand{\Potts}{{\mathfrak{P}}}
\newcommand{\Pcat}{{\mathsf{P}}}
\newcommand{\Mat}{{\mathsf{Mat}}}
\newcommand{\mat}{\left( \begin{array}{ccccccccccccccccccccc}}
\newcommand{\tam}{\end{array} \right)}

\newcommand{\ul}[1]{\underline{#1}} 

\newcommand{\Fk}{\mathbbm{k}}   

\newcommand{\pro}{\mathcal{P}}

\newcommand{\DDD}{{\mathcal{D}}}  

\noindent 
In this paper 
$\N$ denotes the set of  natural numbers and $\N_0$ the natural numbers with 0. 
For $n \in \N$ define $\underline{n} = \{1,2,...,n\}$ and 
$\underline{n}' = \{1',2',...,n'\}$.
The symmetric group is denoted $\Sym_n$. 
We write $\omega_n$ for the Coxeter longest-word element in $\Sym_n$. 
We will need some notation for set partitions.

\newcommand{\bb}[1]{b^{#1}}  
\newcommand{\bbb}[1]{b_0^{#1}}  
\newcommand{\w}[1]{w^{#1}}

If $S$ is a set, we write $\rmP_S$ for the set of set partitions of $S$. 
If $p \in \rmP_S$ and $S' \subset S$ we write $p|_{S'} \in \rmP_{S'}$ for the 
partition restricted to the subset; 
and $\#_{S'}(p) \; := \; |p| - |(p|_{S'})|$ 
(the number of parts of $p$ not intersecting $S'$). 

If $g$ is a graph on vertex set  $V$ we write $\pi_0(g) \in \rmP_V$ for the partition according to connected components of $g$. 

\mdef  \label{pa:Pn}
Consider a set consisting of 
$n+m$ nodes, 
organised as
an `upper' row of $n$ nodes and a lower row of $m$ nodes.
We label the upper nodes, left to right, 
using the set 
$\underline{n} = \{1,2,\dots,n\}$; 
and the lower row similarly by 
$\underline{m}' = \{1',2',\dots,m'\}$. 
We write $\rmP_{n,m}$ 
(or $\rmP_{nm} $ if $nm$ is unambiguous)
for the set of all set partitions of the above set.
That is, 
$\rmP_{n,m} = \rmP_{ \underline{n} \cup \underline{m}' } $. 

\newcommand{\piz}{\pi_0}  

\mdef 
An element of $\rmP_{n,m}$ can be represented 
as a graph $g$ on a vertex set  $V \supseteq \underline{n} \cup \underline{m}' $, 
using $\pi_0(g)|_{  \underline{n} \cup \underline{m}' }$. 
Hence an element of $\rmP_{n,m}$ can be represented 
pictorially
as a graph drawn 
in a rectangular box with the rows of nodes 
$\{ 1,2,...,n\}$ 
on the top and 
$\{ 1',2',...,m'\}$
on the
bottom edge of the box,
as for example here:
\beq 
\qquad 
\{\{ 1,2,6,6'\},\{ 3,4 \}, \{ 5,3'\}, \{ 1',2'\},\{ 4',5'\}\} 
\; = \; 
\rb{-.21in}{\ig[width=3cm]{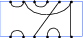}}
\in \rmP_{6,6}
\eq 
Vertices 
in $V\setminus \underline{n} \cup \underline{m}'$ 
(if any) are  
drawn in the interior. 
(Observe that drawn graph edges can cross in the interior of the box, but 
should  
do so transversally, 
and not at a vertex, 
to avoid ambiguity.)
 $\; $ 
Even more schematically we have pictures/diagrams such as:
\beq\label{bn} 
\bb{n} \; 
\; := \; \;
\rb{-.21in}{\ig[width=3cm]{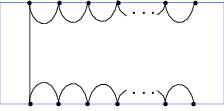} }\in\rmP_{n,n}
\eq
%
and
\[ 
\rb{-.02in}{\ig[width=3cm]{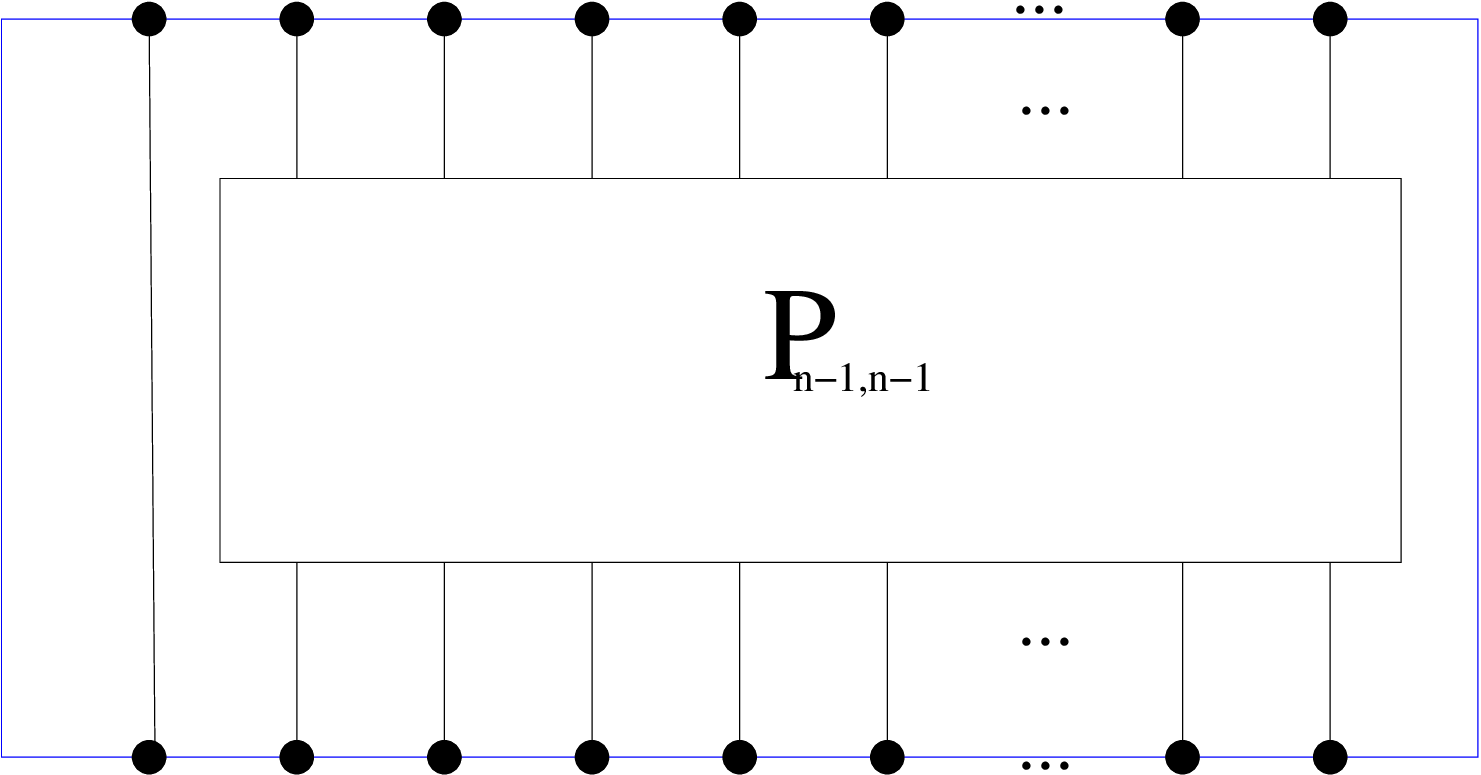}}
\]
representing a certain specific inclusion of $\rmP_{n-1,n-1} \hookrightarrow \rmP_{n,n}$. 
(Of course, such pictorial representations are highly non-unique.
There are many graphs representing the same partition; and many 
embeddings in the box for each graph.)
%

\ignore{ 
We can even allow some representative graphs with additional vertices 
drawn in the interior of the box, so long as the connections among 
$\underline{n} \cup \underline{m'} $  are correct. 
}

\mdef  \label{de:Pcomposition}
Note that in this way a suitable {\em stack} of two 
or more
box representations 
becomes itself another graph and hence a  
representation of another partition. 

Thus for example the following diagram, 
which will be important later (see e.g. \eqref{eq:wpw}), 
indicates a map taking any element of 
$\rmP_{n,n}$ as input and producing an element of $\rmP_{n-2,n-2}$ as output: 

\beq \label{eq:wpw0} 
\rb{-.6422in}{\ig[width= 8.10cm]{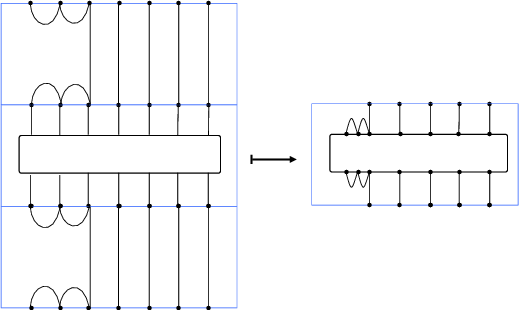}}
\eq

\mdef 
Let $p\in\rmP_{\nk,m}$. 
An element $s\in p$ is called a {\em part} or {\em block}.
A part is called a {\em propagating part} if it 
contains  
nodes from both the upper and lower row. 

For example, the partition 
\rb{-.21in}{\ig[width=3.5cm]{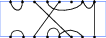}}$\;\in\rmP_{8,8}$ 
has 3 propagating parts and 4 non-propagating parts.

\medskip 

In this Section we follow    
\cite{amm2019tonal} 
closely,   
apart from adding some appropriate fresh examples
(we reproduce some constructions here simply for convenience of reference).

\medskip 

\mdef  \label{de:Pn}
For given ground 
ring $\Fk$,
and $\nk \in \N_0$, 
let $P_\nk(\delta)$ be the usual Partition algebra with parameter $\delta \in \Fk$.
We assume 
some 
familiarity with these algebras (see e.g. \cite{Martin94,Halverson05})
and the corresponding partition category $\Pcat(\delta)$ 
(again as in \cite{Martin94}, or for example \cite{Martin08a}).
\\ 
For $d \in \N$ 
we will define the $d$-tonal partition algebra $\PP$ as a subalgebra 
of $P_\nk(\delta)$, in (\ref{de:P2}).

Here we will fix the ground field to $\C$,
unless stated otherwise
(much of our exposition works for arbitrary $\Fk$, 
but not all of the representation theory, 
and for the latter steps below a different approach would be needed even for $\Fk$  an arbitrary field).

Recall 
that $\rmP_n := \rmP_{n,n}$ 
can be taken as
the defining basis of the algebra $P_\nk(\delta)$. 
The morphism sets 
of the partition category $\Pcat(\delta)$ are 
$\Pcat(\delta)(i,j) = \Fk \rmP_{i,j}$. 
As noted, we assume familiarity with the category composition
(which, when $i=j=n$, is the algebra composition).
In particular the reader will recall that it is a mild generalisation of the 
box-stack composition 
illustrated in (\ref{eq:wpw0})
above,
where a graph $g$ with $c$ purely-interior connected components represents 
$\delta^c \piz(g)$ rather than simply $\piz(g)$.

\mdef  \label{exa:ww}
Examples and notations. 
For $l \in \N$ 
let $\w{l} = \{ \{1,2,...,l \}\}  \in \rmP_{l,0} \;$ and 
$\w{l*} = \{ \{1',2',...,l' \}\} \in \rmP_{0,l}$
(for given $l$ these are written simply $\w{}$ and $\w{*}$ in \cite[(2.4)]{amm2019tonal}). 
These are composable in the category:
$\w{l} \w{l*} = \{\underline{l}, \underline{l'} \}$
and
$\w{l*} \w{l} = \delta \emptyset \; \in \Fk \rmP_{0,0}$.
Let 
$\sigma \; := \;  \{\{1,2'\},\{1',2\}\}   \in \rmP_{2,2} $
and 
$
\;\U \; :=\;  \w{2} \w{2*} \; = \; \{\{1,2\},\{1',2'\}\}   \in \rmP_{2,2}  
$
(in \cite[(2.2)]{amm2019tonal} this is denoted $u$). 
Let 
$ \;  \bb{l} \; := \;
\{\underline{l} \cup \underline{l'} \} $
(as in \cite[(2.2)]{amm2019tonal})
and 
$\; \bbb{l} := \; \w{l} \w{l*} \; = \; 
\{\underline{l} , \underline{l'} \} 
\;  \; \in\rmP_{l,l} $.

\mdef \label{de:flip}
The flip map $-^* : \rmP_{i,j} \rightarrow \rmP_{j,i}$ is obtained by toggling the prime label
on vertices (thus flipping diagrams from top to bottom). 
For example $\{\{1,2,1'\},\{2'\}\}^* = \{\{1,2',1'\},\{2\}\}$;
and see also (\ref{exa:ww}).
Note that this extends linearly to take the algebra to its opposite (and also the category). 

Observe that there is a natural inclusion 
of the symmetric group
$\Sym_n \hookrightarrow P_n(\delta)$
(the elementary transposition in $\Sym_2$ passes to $\sigma$ in (\ref{exa:ww}) above, and so on). 
Recall longest-word $\omega_n$. 
For $p \in \rmP_{n,m} $
the map $p \mapsto \omega_n p \omega_m  \in \rmP_{n,m}$,
denoted $p \mapsto \overline{p}$ 
corresponds to {\em left-right} flip of a diagram for $p$. 

\mdef  \label{pa:U}
Observe that side-by-side concatenation of diagrams gives a map
$\otimes : \rmP_{i,j} \times \rmP_{k,l} \rightarrow \rmP_{i+k,j+l}$.
%
%
For example 
\beq 
\U \otimes \U \; = \; 
\{\{1,2\},\{1',2'\}\} \otimes \{\{1,2\},\{1',2'\}\} 
\; = \;
\{\{1,2\},\{1',2'\}, \{3,4\},\{3',4'\} \}  \; \in \rmP_{4,4} 
\eq 
Then 
$\U^{\otimes m} \in \rmP_{2m,2m}$  $\;$ 
- noting that $\otimes$ is an associative composition. 

\newcommand{\W}{{\mathcal W}} 
\newcommand{\A}{{\mathcal A}}  
\newcommand{\one}[1]{1^{#1}}  

If we write 1 for a partition we mean 
$1=\{\{1,1'\}\}$. 
Then $\one{n} = 1\otimes 1 \otimes ...\otimes 1$ ($n$ factors). 
(In \cite{amm2019tonal} this $\one{n}$ is denoted $1_n$.)  
As in \cite[(3.8)]{amm2019tonal} we have, for given $n$,
\[
W^l \; := \; \left( \w{l} \w{l*}  \right) \otimes \one{n-l}
\; \in \; \rmP_{n,n},
\hspace{.9cm}
W_b^l \; \; := \;  b^{l+1}   \otimes 1^{n-l-1}
\; \in \; \rmP_{n,n}   ,
\hspace{.9cm}
\overline{W}_b^l \; \; = \;   1^{n-l-1}   \otimes b^{l+1}  
\]
For example the top `factor' on the left in (\ref{eq:wpw0}) is $W_b^2$ 
(in case $n=7$).


For given $n>0$ let 
$\A_1 \; := \; \{\{1\},\{1'\}\}\otimes 1^{n-1} =W^1$.

\mdef
Recall that 
this 
$\otimes$  extends to  
the usual 
monoidal product on category $\Pcat(\delta)$.




\mdef  \label{de:Pcat}
The object monoid of monoidal category 
$\Pcat(\delta)$   
is $(\N_0, +)$. 
\ignore{ 
For $i,j \in \N_0$ we write $\Pcat_{ij}$ for the set of set partitions of 
the set $\ul{i}\cup\ul{j}'$,
where $\ul{n}=\{1,2,..,n\}$ and $\ul{n}' = \{1',2',..,n'\}$.
}
As noted in (\ref{pa:U}), 
the monoidal composition is given by `side-by-side' concatenation. 
Note in particular that
the `source' object $n$ corresponds to the set $\{1,2,..,n\}$,
so that 
the monoidal product 
$1\otimes 2 = 3$ passes to $\{1\} \otimes \{1,2\} = \{1\}\cup\{2,3\}$. 
(In fact the monoidal structure we put on 
$\Pcat(\delta)$, 
while convenient, 
is not natural. It puts an order on the set $\{1,2,..,n\}$ which is not intrinsic to the set in this context.)

For fixed (but arbitrary) 
$\delta$  
we may write just $\Pcat$ for 
the category 
$\Pcat(\delta)$. 
If 
$\Fk$ is fixed but 
$\delta$  
is  {\em not} 
fixed then $\Pcat$ means the category over the ring $\Fk[\delta]$.

\subsection{The tonal partition algebras $P_\nk^\dell(\delta)$} 
\label{ss:constructPd}
$\;$ 

\newcommand{\tone}{tonal}  

\mdef 
Fix a positive integer $\dell$. 
Let partition $p = \{p_1, p_2, \ldots, p_l \} \in \rmP_{n,m}$.
We say $p$ is a  \emph{$\dell$-\tone\ partition} if for each 
part $p_j\in p$ the 
difference between the number of nodes on the top row and the number of nodes on the bottom row of $p_j$ is $0$ modulo $\dell$. 

Let 
$\rmP_{\nk,m}^\dell \; = 
\{p\in \rmP_{n,m} \mid p\text{ is a }\dell\text{-\tone\ partition}\}$.
And let $\rmP_{n}^\dell := \; \rmP_{n,n}^\dell$. 

\noindent 
Examples: 
\\ 
1. the partition 
\rb{-.1in}{\includegraphics[width=2.2cm]{exatonal.eps}}$\;\in \rmP_{6,6}$ 
is a $2$-tonal partition but not $3$-tonal; 
\\
2. the $P_\nk(\delta)$ identity,
$1_n$ 
is in 
$\rmP^\dell_\nk$ for any $\dell$;
\\
3. the partition 
\beq \label{eq:P29ex} 
\rb{-.22in}{\ig[width= 4.5cm]{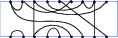}}
\;\; \in \;\; \rmP^2_{9}
\eq 
4. the partition $\U^{\otimes m} $ 
is in $\rmP^2_{2m,2m}$.

\mdef    \label{de:P2}
The restriction of the partition algebra $P_\nk(\delta)$  to the span of $\dell$-tonal partitions $\rmP_\nk^\dell$  yields 
a subalgebra --- see e.g.~\cite{amm2019tonal,Kosuda08,orellana2007partition, tanabe1997centralizer}.
This is called the  \emph{$\dell$-tonal partition algebra}, 
and denoted $\PP$ 
(various other names are also used in the literature).
We write $\Pcat^{\dell}(\delta)$ for the corresponding category
\cite{amm2019tonal}.

\mdef 
We write $\rmPe_S$ 
(respectively $\rmPe_n$)
for the subset of $\rmP_S$ 
(respectively $\rmP_n$)
of set partitions where every
part is of even size.
Observe that $\rmP^2_n = \rmPe_n$.

\mdef  \label{pa:Wdnil}
Observe that for all $d,m \in \N$ the element $W^{dm} \in P^d_{dm}$ 
is 
nilpotent when $\delta=0$, 
and 
indeed generates a nilpotent ideal. 
Thus the algebra $P^d_{dm}(0)$ is not semisimple.



\ignore{{ 
\mdef\ca{Shall I remove 2.4?? I do not think we need this part anymore} 
We denote the set of $\dell$-tone partitions , i.e. diagram basis of $\PP$, by $\mathrm{P}_\nk^\dl.$

I.e. if we call a connected component of a diagram a \emph{block},
then every block in a diagram in $\PP$ has the same number of nodes of
the top and the bottom modulo $\dell$.
}}


\ppmy{ 
\subsection{Preparatory homological and combinatorial generalities}
$\;$ 
\ppm{[probably just let 1st two items here be at end of previous subsection; and 3rd item be at start of next subsection, so this subsection disappears!]}
}

\newcommand{\siz}{\lfloor}

\mdef 
Given a set $S$ of finite sets, 
and $Q \in \N$,
we write $S\siz_Q$ for the subset of sets of order $Q$. Thus
\beq 
S \;=\;\; \bigsqcup_Q \; S\siz_Q
\eq

\mdef  \label{de:bell1}
Recall the Bell numbers 
$B(n) = | \rmP_{n,0} | = |\rmP_{\underline{n}}|$
and the Stirling numbers (of the second kind) 
$S(n,l) = | \rmP_{n,0}\siz_l | $
(see e.g. \cite[Ch.8]{Martin91} or \cite{Martin94} for realisations in our setting, and 
\cite{benkart2012planar,benkart2016schur,benkart2017walks} and
references therein). 
We have 
\(
B(n) = \sum_l S(n,l)   .
\)

Analogously let 
$$
\T(m,l) \;  :=  \;
       \left|    
             \rmPe_{2m,0}\siz_l \right|  \;\;
$$
- i.e. the number of even partitions into $l$ parts of the set $\underline{2m}$.

\subsection{\ppmm{Basic representation theory}} \label{ss:brt}       $\;$

\mdef  \label{de:GFfunctors}
Recall that for an algebra $A$ and idempotent $e \in A$ 
we have the following
`idempotent reciprocity' 
between the module categories of the algebras $eAe$ and $A$ 
(see e.g. 
\cite{green2007polynomial,Martin94,amm2019tonal}). 
\\ 
The functor 
$G_e : eAe-\!\!\!\!\!\mod \rightarrow \; A\!-\!\!\!\!\!\mod$ 
is given 
on objects
by 
$G_e - \; = \; Ae \otimes_{eAe} -$;
and this has a right-adjoint given by $F_e - = e- \;$.
The maps on 
morphisms will be clear, 
see e.g. \cite{CurtisReiner62,Martin91}.
\\ 
Functor $G_e$ is right-exact and $F_e$ is exact.
\\ 
Now suppose the ground ring is a field.
If $L$ is a simple $A$-module and $eL \neq 0$ then 
$eL$ is a simple $eAe$-module and 
the Jordan--Holder composition multiplicity 
$ \;  [eM : eL] \; = \;  [M : L] \; $
for $ \; M \in A\!-\!\!\!\!\!\mod$.
\\
If the ground ring is a field and 
$\Lambda(A)$ denotes an index set for simple $A$-modules then there is a choice for the index set $\Lambda(A/AeA)$ such that 
$\Lambda(A) = \Lambda(eAe) \sqcup \Lambda(A/AeA)$.

\medskip



\ignore{{ 

\subsection{How things work for the ordinary partition algebra}   $\;$ 

In this paper we follow the original method that works for
the partition algebras themselves.
In this subsection we assume familiarity with the partition algebra, and can then give a convenient brief overview of our approach in terms of the partition algebra case.

}}


\newcommand{\Cdelta}{\C[\delta_c\in \C]} 

Recall that the partition algebras $P_n$ 
(for $n \in \N$)
may be defined over $\Z[\delta]$, where they are of finite rank; and hence over any ring that is a $\Z[\delta]$-algebra, by base change.
For example we can work over the ring $\C[\delta]$, 
which is a $\Z[\delta]$-algebra by inclusion; 
or over $\C$ by evaluating
$\delta$ as some element $\delta_c\in\C$
(we may write   
$\Cdelta$ 
for $\C$ made a $\Z[\delta]$-algebra in this way, 
for example $\C[2]$,
or simply write $\C$ if no ambiguity arises, leaving the action of $\delta$ implicit). 

\ppmx{[just around here is where the standard modules are defined. this is one of the things that we will parallel for the tonal case, so perhaps this should be highlighted more clearly - so that explicit reference to this part can be made later. - and then explicitly paralleled in the tonal case. (BTW, I note that here it says very explicitly that qh will not be used.)]}

Indeed 
the pair 
$(\C[\delta],  \Cdelta )$
represent a modular system 
(in the sense, for example, of 
\cite{CurtisReiner62} or \cite[\S1.9]{Benson95})
for the 
representation theory of 
$P_n(\delta_c\in\C)$ - the algebra over 
$ \Cdelta $.    

\ppmy{[A copy of Pn-Brat-fig  was 
invoked here (it is labeled as fig:PnBratelli, which unpacks as \ref{fig:PnBratelli}, due to the later duplicate, but here becomes Fig.1). But without comment! And it is still duplicated later! Need to resolve this. Perhaps put the fig in Intro somehow -- now DONE! (as PnBratFig.tex) -- and remove other instances?]}

\mdef \label{lem:ss01}
Over the field of fractions of 
$\C[\delta]$
the algebras $P_n$ are semisimple. 
The simple modules can be labeled with $\lambda\in \Lambda^n$, the set of integer partitions of ranks up to $n$
(see
Fig.\ref{fig:PnBratelli}, 
and
\ref{pa:globloc102}).
\ppmy{[--maybe rearrange so   \ref{pa:globloc102} arrives here-ish??]}
There are integral lattice forms
of the simple modules,
a preferred choice of 
which 
(see \ref{pa:globloc102})
thus define
{\em standard modules} 
over $\Z[\delta]$, and hence by base change over any ground ring.
$\;$ 
\\
(At each step in this 
review
section we  
give, for ease of comparison, a link to the lift of the step  
to the 
new
$P^2$ case in \S\ref{ss:2tonal}. Here then:
\confer{\ref{pa:Pn2mod}}
)


Once the ground ring is fixed, the standard modules may be 
denoted $\SSS_n(\lambda)$, where label $\lambda$ is an integer partition of rank up to $n$.
Over $\C[\delta]$ these modules are quotients of indecomposable projectives.
Upon specialisation to $\delta_c \in \C^\times$ they have  
pairwise distinct simple heads.
We write $\DDD_n^{\delta_c}(\lambda)$ for the simple head of $\SSS_n(\lambda)$  in the specialisation evaluating $\delta$ at $\delta_c$. 

The case $\delta_c =0$ can be treated as a mild degeneration of this setup, and we will only write out these details explicitly when needed for our main proof.

\medskip 

{\it{Remark}}: Here we use the language of classical Brauer modular representation theory (see e.g. Curtis--Reiner \cite[\S 18]{CurtisReiner62}). 
The modules $\SSS_n(\lambda)$ 
also coincide with cellular cell modules and quasi-hereditary standard modules but we will not need this technology.
A natural construction is via the partition category $\Pcat$ 
over the corresponding ground ring, which we briefly recall in \S\ref{ss:PottsFunctors}.

\ppmx{[could perhaps put \ref{pa:globloc101} here for example. Done.]}




\mdef   \label{pa:globloc101}
\ppmx{[before, we had the rep theory we need demoted to the start of sec.\ref{ss:review}. but now we are putting some of it here (15/8/24: now moved to sec.3) - so that we can parallel the index set setup when we come to the $P^2$ story below. maybe we should therefore move more of it from sec.\ref{ss:review} to here? or otherwise move the corresponding bit for $P^2$ to later?]}
Note that if 
our parameter 
$\delta \in \Fk$ is invertible 
and $n>0$ 
then $\A_1$ 
from (\ref{pa:U}) 
is a pre-idempotent 
(idempotent up to a scalar)
in $P_n$. 
Observe 
that 
$\A_1 P_n \A_1 \cong P_{n-1}$; and that the quotient 
$P_n/ P_n \A_1 P_n \cong \Fk \Sym_n$,
and so is semisimple over $\Fk = \C$. 
\confer{\ref{pa:globloc2}}

\mdef  \label{pa:globloc102}
Let $\Fk = k$ be a field 
and let  $ \Lambda(k\Sym_n) $ denote the usual index set for simple $k \Sym_n$-modules. 
For example 
if $k=\C$ this is the set $\Lambda_n$ of integer partitions of $n$. 
It 
then
follows 
from (\ref{pa:globloc101}) and (\ref{de:GFfunctors}) 
(or see e.g. \cite{Martin94} or \cite[(4.1)]{amm2019tonal}) that
$\Lambda(P_n) 
= \Lambda(P_{n-1} )  \cup \Lambda(k \Sym_n)
= \cup_{i=0}^n \Lambda(k \Sym_i)
$
is an index set for simple $P_n$-modules
(provided $\delta \neq 0$). 
If $\delta \neq 0$ and  $k=\C$ we have $\Lambda(P_n) =\Lambda^n \; := \; \cup_{i=0}^n \Lambda_n$.

For $\lambda \in \Lambda(P_n)$ we write $|\lambda|$ for its 
order as an integer partition, i.e. the $i$ value such that     $\lambda \in  \Lambda(k \Sym_i)$.
This `rank' places $\lambda$ in the above decomposition of $\Lambda(P_n)$. 

\newcommand{\GA}{G_\A}  

In case $k$ contains $\QQ$ (e.g. $k=\C$)  
the standard $P_n$-modules may be given as follows. 
For rank $n$ then $\SSS_n(\lambda)$ is 
given by 
the corresponding simple module of $P_n/P_n \A_1 P_n$
(for which, over $k=\C$, all lattice forms are equivalent; but we can take it to be the corresponding $k \Sym_n$ Specht module).
If the rank is less than $n$ then we can proceed as follows.
If $\delta\neq 0$,
then we can take 
$\SSS_n(\lambda)  = G_{\A_1} \SSS_{n-1}(\lambda)$.
Or if $n>1$ then we can take 
$\SSS_n(\lambda)  = G_{W^1_b} \SSS_{n-1}(\lambda)$
(or indeed use other standard constructions that are equivalent where defined, for example using the category $\Pcat$ as a source of bimodules - we will not need these details here).
$\;$
\confer{\ref{pa:Trep}}

\mdef  \label{de:GA}
Consider the setup in (\ref{de:GFfunctors}). 
In our case, with $A=P_n$ and $e= \A_1$ let us 
now write simply $G_\A$ and $F_\A$ for the functors. 
Pass to any specialisation as in (\ref{lem:ss01}).
Observe that $\A_1 L=0$ if the index of 
simple module 
$L$ obeys $|\lambda| =n$; and is nonzero otherwise.
It follows in particular that 
we have `upper-triangularity'
of standard module decomposition matrices, 
with simple composition multiplicities obeying:
\begin{equation}
      \label{eq:uppertriangular}
[\SSS_n(\lambda) : \DDD_n^{\delta}(\mu)] \neq 0 
\;\implies\; 
|\lambda| < |\mu |  \; \mbox{ or } \lambda=\mu
\end{equation}
%
%
{From this we see that $\Lambda(P_n)$ has a 
natural partial order with a 
unique `first' element, $\lambda=\emptyset$. We call $\SSS_n(\emptyset)$ the `spine' module 
--- it occurs in the spine 
- the left edge -
of Fig.\ref{fig:PnBratelli} 
for example.}
$\;$
\confer{\ref{pa:P2head}}


\newcommand{\E}{E}

\medskip 
\mdef   \label{de:Pn-spine} 
Fix $n$  and a ground field $k$ (with the property of $\Z[\delta]$-algebra).
Let $\E_1 =\bb{n} \in P_n$  
as in (\ref{bn}) above. 
And, as in (\ref{exa:ww}), let 
\beq   \label{eq:E0} 
\E_0 = \bbb{n}  = \{ \{1,2,...,n\},\{1',2',...,n'\}\}
\eq 
We  have $\E_0 \E_0 = \delta \E_0$, 
and 
\beq \label{eq:EPE} 
\E_0 P_n \E_0 = \delta k \E_0 ,
\eq 
Thus $\E_0$ passes to an unnormalised primitive idempotent in any specialisation for which $\delta$ has an inverse.
In these cases then, the left ideal $P_n \E_0$ and right ideal $\E_0 P_n$ are  indecomposable projective. 

In fact,
noting (\ref{pa:globloc102}), 
we can take 
$\SSS_n(\emptyset) = P_n \E_0$ up to isomorphism; and 
$\SSS_n((1)) = P_n \E_1 / P_n \E_0 P_n$. 

By (\ref{eq:EPE}) we have a 
multiplication map
$\mu :  \E_0 P_n \times P_n \E_0  \rightarrow k\E_0$
and then a map $\kappa :  \E_0 P_n \times P_n \E_0  \rightarrow k$
given by 
$\mu(a,b)=ab = \kappa(a,b) \E_0$. 
Observe that this gives a contravariant form on $P_n \E_0$
with respect to the usual flip map from (\ref{de:flip}). 

Recall the chain of ideals
$P_n \supset P_n E_1 P_n \supset P_n E_0 P_n$.
We immediately have $E_1 E_1 = E_1$ and 
$E_1 P_n E_1 = k E_1$ mod. $P_nE_0 P_n$.
Thus $E_1$ is a primitive idempotent in the corresponding quotient. 
\confer{\ref{de:Pn2-spine}}

\mdef 
\label{de:PnE0basis}
Observe that $P_n E_0$ has basis the subset of the partition basis of $P_n$ consisting of all partitions in which the bottom row of vertices exactly make one of the parts. The basis is thus constructed by taking all partitions of the top row of $n$ vertices. 
There is a corresponding `flipped' basis for $\E_0 P_n$,
and with these bases we fix a gram matrix $\Gram_n(\emptyset)$ for the form $\kappa$. 
(See also \S\ref{ss:gramgen00}.)

\medskip 

In (\ref{pa:2bas})
it will be convenient to organise the top-row partitions 
in the partition basis of $P_n \E_0$ 
according to the number of parts.
There is of course one with one part, for any $n$. 
The partitions into two parts or less are 
double-counted by   
the power set of the set of vertices, so there are $2^{n-1}$ of these. 
For the three parts case, and higher cases, things are a bit more complicated, 
see (\ref{pa:floor1}) {\em et seq}. 



\subsection{Gram matrix generalities} \label{ss:gramgen00} 

$\;$

Before proceeding, let us take a moment to abstract the core gram matrix arguments 
from the ordinary $P_n$ case to a level that includes also the tonal cases.

\newcommand{\form}[1]{\langle #1 \rangle}
\newcommand{\bass}{\boldsymbol{\beta}}
\newcommand{\Gramb}{\Gram_{\bass}}
\newcommand{\rank}{{\mathrm{rank}}}
\newcommand{\LL}[1]{L^{#1}}
\newcommand{\head}{{\mathrm{head}}}

Here we start by briefly recalling some generalities (cf. e.g. \cite{amm2019tonal}). 
Throughout this section, the $P_n$ case will provide examples.

\newcommand{\herpre}{heredity preidempotent}

\mdef \label{pa:gram00}
Let $K$ be a commutative ring, and let $A$ be a finite rank $K$-algebra. 
Suppose $A$ has  an involutive 
self-set-map, written $a \mapsto a^*$, that gives an antiautomorphism
(or equivalently an opposite-isomorphism) --- such as the transpose map for a simple matrix algebra. 
If there 
is a non-zero element $e$ in $A$ such that $e^* = e$ and 
\beq   \label{eq:eAe=Ie}
eAe = Ie
\eq  
for some ideal $I \subseteq K$, then we call $e$ a {\em heredity preidempotent}.

Caveat: a \herpre\ is not necessarily a pre-idempotent in the sense of (\ref{pa:globloc101}).
%
{(Recall that an idempotent $e\in A$ is primitive if and only if $eAe$ is local.}
\footnote{ \label{eg:cyclic}
\newcommand{\za}{\zeta}
Aside: Note that the condition $e=e^*$ on an
unnormalised primitive idempotent is significant. 
It does not hold, for example,
for primitive idempotents of 
the complex group algebra of the cyclic group of order 3, 
$\; G = \; \langle c | c^3 \rangle$,
with map $-^*$ given by inverting elements of the group basis. 
Consider $e = c^0 + \za c +\za^2 c^2$ where $\za^3 = 1 \in \C$,
from which $e^* = c^0 +\za^2 c +\za c^2$.}%
)

Suppose $e$ is a heredity preidempotent, and 
in particular that $ee=\kappa_e e$ for some $\kappa_e \in I$. 
Then the left-ideal $Ae$ has a   
form 
on it
$$
\langle -,- \rangle_e : \; Ae \times Ae  \;\; \rightarrow \;\; I
$$ 
given by 
\beq \label{form} 
\form{ae,be}_{{e}}e  \; =\; (ae)^* be \; = ea^* b e \in Ie  . 
\eq 
Observe from the construction that this form is bilinear and indeed contravariant
with respect to map $*$. 
(Note that this form makes sense when $e \neq e^*$ but, for example, in our cyclic group 
example above$^{\ref{eg:cyclic}}$ 
the form vanishes, corresponding to the module $Ae$ being simple but not isomorphic
to its contravariant dual.)

Let $\bass$ be a $K$-basis for $Ae$. 
Write $\rho$ for the matrix representation corresponding to 
$Ae$ afforded by $\bass$ 
(hence $\rho^T$ is also a representation,
\ppmm{by the opposite-isomorphism property});
and 
write 
$\Gram_{\bass}$ for the gram matrix of $\form{,}_e$ 
with respect to $\bass$. 
Observe that $ \Gram_{\bass}$ is a $K$-matrix, hence with $K$-valued determinant.
And
\beq  \label{eq:intert2} 
\rho^T(a) \;\; \Gram_{\bass} \;\; = \;\; \Gram_{\bass} \;\; \rho(a) \hspace{1cm} \forall a \in A
\eq 

\mdef
Example. Let $K=\Z[\delta]$ and consider the algebras $P_n$ over $K$. 
Then $E_0$ as in (\ref{de:Pn-spine}) is a \herpre.
We have a basis for $P_n E_0$ as in (\ref{de:PnE0basis}), and in particular for $n=3$ we can order
as $\{ (1)(2)(3), (12)(3), (13)(2), (1)(23), (123) \}$ (omitting the primed part). 
See (\ref{eg:gramS30}) for explicit construction of the matrix representation; and Fig.\ref{fig:gram234} 
for the corresponding gram matrix. 
\ppmy{...
\ppm{[why ... here?]}
}

\mdef 
Now let $k$ be an algebraically closed field that is a $K$-algebra by a map $\psi:K \rightarrow k$
(for example $\C$ is a $\C[\delta]$-algebra by letting $\delta$ act as some $\delta_c \in \C$),
and denote $k$ as $k_{\psi}$ when regarded as this $K$-algebra.
Thus $A(\psi) =  k_{\psi} \otimes_{K} A $  is a finite dimensional $k$-algebra. 
Suppose the image of $\bass$ is a basis for the image of $Ae$
--- hereafter we call such images `evaluations'. 
And suppose the evaluation of $ \Gram_{\bass} $ intertwines the evaluations of representations
$\rho $ and $\rho^T$, which are $A(\psi)$-representations. 

Now suppose in particular that $K=\Z[\delta]$; and $k= \C$, so that 
specific 
$\psi$ maps are obtained by 
$\psi: \delta \mapsto \delta_c \in \C$. And write $A(\delta_c)$ for this $A(\psi)$, and so on. 

Since 
the determinant 
$|\Gramb| \in \Z[\delta]$ 
we may now define exponents $\alpha_{\delta_c}$ by writing 
\beq  \label{eq:gramfactor2}
|\Gramb |  \; = \; \varkappa \prod_{\delta_c} (\delta-\delta_c)^{\alpha_{\delta_{c}}} 
\eq 
(note $\varkappa \in \Z$) 
--- factorising if necessary over $\C$ by the algebraic closure of $\C$.

\mdef \label{pa:delta=0}
Observe that if matrix 
$\Gramb(\delta_c) = 0$    
then every matrix element in $\Gramb$ 
can be written in the form $(\delta-\delta_c)x$ for some $x \in \C[\delta]$,
so 
there is a $\C[\delta]$-valued matrix $ \Gramb' $ such that 
$\Gramb= (\delta-\delta_c) \Gramb'$ with 
$\Z[\delta] \hookrightarrow \C[\delta]$ in the 
natural way. 
And 
$\rho^T \; \Gram_{\bass}' \; =\; \Gram_{\bass}' \; \rho$.
Iterating 
this $\Gramb\rightarrow\Gramb'$ process, 
we arrive at 
a matrix 
$\Gramb^+$ with $\Gramb^+(\delta_c  )  \neq 0 $.
{And hence we have a non-trivial intertwiner at $\delta=\delta_c$.}

Aside: 
We would like (but here don't need) to deduce from (\ref{eq:eAe=Ie}) that $Ae$ has a simple head 
over $\C_{\delta_c}$,
regardless of whether $e$ is a pre-idempotent. 

\mdef \label{pa:argument-for-4.8}
In a case of $\delta_c$ 
as above 
where $e$ is pre-idempotent in $A(\delta_c)$,
i.e. $\kappa_e(\delta_c) \neq 0$, 
then $Ae \cong Ae'$ for an idempotent 
$e' =   \kappa_e(\delta_c)^{-1}  e$, 
and indeed $e'$ is primitive so $Ae$ has simple head, call it $L^{\delta_c}$.
Observe from (\ref{eq:eAe=Ie}) that 
$Trace(\rho(e))= Trace(L^{\delta_c}(e)) = \kappa_e  \;$ 
and 
$\; Trace(L_\chi(e))=0$ for every other simple module. 

Indeed 
in all such cases
$\Hom(Ae,Ae) \cong eAe \cong k$ as a space (and $eAe^{op}$ as an algebra), 
while for any module $M$ we have 
$\dim(\Hom(Ae,M)) = [M : \head(Ae)]$ (the composition multiplicity in $M$ of the simple head),
so $Ae$ has no other factor isomorphic 
to the head.
Observe that $e\;\head(Ae)\neq 0$ and $e$ acts like 0 on any other simple factor of $Ae$.
Since $e=e^*$ this also holds on the right - meaning every simple factor 
of right module $eA$ except the head is 
killed by $e$, so the head of $Ae$ and the socle of its contravariant dual are isomorphic,
and this simple module does not appear elsewhere in the composition factors of either module.
So by Schur's Lemma the intertwiner must map head of $\rho$ to socle of contravariant dual
$\rho^T$.

In these cases we 
deduce that 
$$
\rank(\Gramb^+(\delta_c)) = \dim \LL{\delta_c} 
$$ 
(intertwiner maps head to socle)
so exactly 
$\dim(Ae) - \dim(\LL{\delta_c}) $  
diagonal entries in 
the Smith form
$Sm(\Gramb)$ 
(we assume $\Gramb$ has a Smith form, as it has for example if
$ K$ is a PID)
have at least one factor $(\delta-\delta_c)$.
Comparing with (\ref{eq:gramfactor2}) we deduce
the following.

\mdef {\bf Proposition}. 
Suppose $A$ is a finite rank $\Z[\delta]$ algebra and $e$ a heredity preidempotent.
Then for each $\delta_c \in \C$ with $\kappa_e(\delta_c) \neq 0$ we have 
exponent 
\beq  \label{eq:alphabound2}
\alpha_{\delta_c}   \; \geq \;  \dim(Ae) - \dim(\LL{\delta_c})     
\eq
(from (\ref{eq:gramfactor2}))
with strict equality to zero if $ \LL{\delta_c} = Ae $.
%
\label{pr:hid22}
In particular 
if $ \LL{\delta_c} \neq Ae $,
that is if $\Gramb(\delta_c)$ is singular, then $Ae$ is 
indecomposable but 
not simple and $A$ is not semisimple;
and 
if   $\Gramb(\delta_c)$ is non-singular then $Ae$ is simple. 
\qed

\medskip 

\mdef We write $\Gram_n(\emptyset)$ 
or $\Gram_n(0)$ 
for the Gram matrix 
of the $P^{}_n(\delta)$ 
module   $\mathcal{S}_n(\emptyset)$ 
in the   {standard-basis (as in (\ref{de:PnE0basis}))}.

\begin{eg} \label{eg:gramSn0}
\ignore{ 
Consider the  
Gram matrix 
$\Gram_n(\emptyset)$
of the $P^{}_n(\delta)$ 
module  $\mathcal{S}_n(\emptyset)$ 
in the  {standard-basis (as in (\ref{de:PnE0basis}))}.
}
The Gram matrix $\Gram_n(\emptyset)$
is 
given 
for $n=2,3,4$
in Fig.\ref{fig:gram234}.
For example the notation $(12)(3)$ represents the
partition $\{\{1,2\},\{3\}\}$, and so on. 
(Remark. The only basis element here not present in the
planar/non-crossing/Temperley--Lieb case is $(13)(24)$.) 

\begin{figure} 
$\begin{pmatrix}
\delta^2 & \delta\\
\delta & \delta\\
\end{pmatrix}$ 
$
{\small 
\begin{pmatrix}
\delta^3 & \delta^2 & \delta^2 & \delta^2 & \delta\\
\delta^2 & \delta^2 & \delta & \delta & \delta\\
\delta^2 & \delta & \delta^2 & \delta & \delta\\
\delta^2 & \delta & \delta & \delta^2 & \delta\\
\delta & \delta & \delta & \delta & \delta\\
\end{pmatrix}
}
$ 
%
$
{\small 
\begin{pmatrix}
\delta^4 & \delta^3 & \delta^3 & \delta^3 & \delta^3 & \delta^3 & \delta^3 & \delta^2 & \delta^2 & \delta^2 & \delta^2 & \delta^2 & \delta^2 & \delta^2 & \delta \\
\delta^3 & \delta^3 & \delta^2 & \delta^2 & \delta^2 & \delta^2 & \delta^2 & \delta^2 & \delta & \delta & \delta^2 & \delta^2 & \delta & \delta & \delta \\
\delta^3 & \delta^2 & \delta^3 & \delta^2 & \delta^2 & \delta^2 & \delta^2 & \delta & \delta^2 & \delta & \delta^2 & \delta & \delta^2 & \delta & \delta \\
\delta^3 & \delta^2 & \delta^2 & \delta^3 & \delta^2 & \delta^2 & \delta^2 & \delta & \delta & \delta^2 & \delta & \delta^2 & \delta^2 & \delta & \delta \\
\delta^3 & \delta^2 & \delta^2 & \delta^2 & \delta^3 & \delta^2 & \delta^2 & \delta & \delta & \delta^2 & \delta^2 & \delta & \delta & \delta^2 & \delta \\
\delta^3 & \delta^2 & \delta^2 & \delta^2 & \delta^2 & \delta^3 & \delta^2 & \delta & \delta^2 & \delta & \delta & \delta^2 & \delta & \delta^2 & \delta \\
\delta^3 & \delta^2 & \delta^2 & \delta^2 & \delta^2 & \delta^2 & \delta^3 & \delta^2 & \delta & \delta & \delta & \delta & \delta^2 & \delta^2 & \delta \\
\delta^2 & \delta^2 & \delta & \delta & \delta & \delta & \delta^2 & \delta^2 & \delta & \delta & \delta & \delta & \delta & \delta & \delta \\
\delta^2 & \delta & \delta^2 & \delta & \delta & \delta^2 & \delta & \delta & \delta^2 & \delta & \delta & \delta & \delta & \delta & \delta \\
\delta^2 & \delta & \delta & \delta^2 & \delta^2 & \delta & \delta & \delta & \delta & \delta^2 & \delta & \delta & \delta & \delta & \delta \\
\delta^2 & \delta^2 & \delta^2 & \delta & \delta^2 & \delta & \delta & \delta & \delta & \delta & \delta^2 & \delta & \delta & \delta & \delta \\
\delta^2 & \delta^2 & \delta & \delta^2 & \delta & \delta^2 & \delta & \delta & \delta & \delta & \delta & \delta^2 & \delta & \delta & \delta \\
\delta^2 & \delta & \delta^2 & \delta^2 & \delta & \delta & \delta^2 & \delta & \delta & \delta & \delta & \delta & \delta^2 & \delta & \delta \\
\delta^2 & \delta & \delta & \delta & \delta^2 & \delta^2 & \delta^2 & \delta & \delta & \delta & \delta & \delta & \delta & \delta^2 & \delta \\
\delta & \delta & \delta & \delta & \delta & \delta & \delta & \delta & \delta & \delta & \delta & \delta & \delta & \delta & \delta \\
\end{pmatrix}
}
$ 
\caption{Gram matrices for $\SSS_n(0)$ for $n=2,3,4$. 
The bases are $\{ (1)(2), (12) \}$,
$\{ (1)(2)(3), (12)(3), (13)(2), (1)(23), (123) \}$,
$\{  (1)(2)(3)(4), (12)(3)(4), (13)(2)(4), 
$ $(1)(23)(4), (14)(2)(3), (1)(24)(3), (1)(2)(34),
(123)(4), (124)(3), (12)(34), ..., (1234)  \}$ respectively (see main text for notation). 
\label{fig:gram234}}
\end{figure}

The corresponding  determinants are:
$
\; \det(\Gram_2(\emptyset)) = 
\delta^2(\delta-1)$; $\;$ 
$
\det(\Gram_3(\emptyset)) =
\delta^5(\delta-1)^4(\delta-2)$; $\;$
$
\det(\Gram_4(\emptyset)) =
\delta^{15}(\delta-1)^{{14}}(\delta-2)^7(\delta-3)$.

\end{eg}

\newcommand{\bassoo}[1]{
\mat 
#1 (1)(2)(3) \\ #1 (12)(3) \\ #1 (13)(2) \\ #1 (1)(23) \\ #1 (123) \tam
}
\newcommand{\basss}{
\mat 
\delta^3 & \delta^2 & \delta^2 & \delta^2 & \delta\\
\delta^2 & \delta^2 & \delta & \delta & \delta\\
\delta^2 & \delta & \delta^2 & \delta & \delta\\
\delta^2 & \delta & \delta & \delta^2 & \delta\\
\delta & \delta & \delta & \delta & \delta\\
\tam 
}

\mdef The Gram matrix $\Gram_n(0)$ intertwines 
the corresponding matrix form of 
$\SSS_n(0)$ with its contravariant dual.
The modules are generically isomorphic so the intertwiner is invertible
in generic specialisations.
But for $\delta=1$ we see from Fig.\ref{fig:PnBratelli} that 
$\SSS_2(0)$
has a 1d submodule, so the intertwiner maps the 1d head to the socle, and so has
rank 1 on evaluation at $\delta = 1$.

\mdef \label{eg:gramS30}
The matrix representation 
of left-module 
$\SSS_3(0)$ with the given ordered basis is given by
\[
(\A_1)\mat 
(1)(2)(3) \\ (12)(3) \\ (13)(2) \\ (1)(23) \\ (123) \tam
= \mat 
\A_1(1)(2)(3) \\ \A_1 (12)(3) \\ \A_1 (13)(2) \\ \A_1 (1)(23) \\ \A_1 (123) \tam
=\mat 
\delta (1)(2)(3) \\ (1)(2)(3) \\  (1)(2)(3) \\ \delta (1)(23) \\  (1)(23) \tam
=\mat 
\delta \\ 1&0 \\ 1&&0 \\ &&&\delta \\ &&& 1 &0 \tam
\mat 
(1)(2)(3) \\ (12)(3) \\ (13)(2) \\ (1)(23) \\ (123) \tam
\]
\[
(A_{12}) \mat 
(1)(2)(3) \\ (12)(3) \\ (13)(2) \\ (1)(23) \\ (123) \tam
=  
\bassoo{A_{12}} 
=\mat 
 (12)(3) \\ (12)(3) \\  (123) \\  (123) \\  (123) \tam
=\mat 
0&1 \\ &1 \\ &&0&&1 \\ &&&0&1 \\ &&&& 1  \tam
\mat 
(1)(2)(3) \\ (12)(3) \\ (13)(2) \\ (1)(23) \\ (123) \tam
\]
\[
(\sigma_{23})\mat 
(1)(2)(3) \\ (12)(3) \\ (13)(2) \\ (1)(23) \\ (123) \tam
=  
\bassoo{\sigma_{23}} 
=\mat 
 (1)(2)(3) \\ (13)(2) \\  (12)(3) \\  (1)(23) \\  (123) \tam
=\mat 
1& \\ &0&1 \\ &1&0&& \\ &&&1& \\ &&&& 1  \tam
\mat 
(1)(2)(3) \\ (12)(3) \\ (13)(2) \\ (1)(23) \\ (123) \tam
\]
\[
(\sigma_{12})\mat
(1)(2)(3) \\ (12)(3) \\ (13)(2) \\ (1)(23) \\ (123) \tam
=  
\bassoo{\sigma_{12}} 
=\mat 
 (1)(2)(3) \\ (1)(23) \\  (12)(3) \\  (13)(2) \\  (123) \tam
=\mat 
1& \\ &1&& \\ &&0&1& \\ &&1&0& \\ &&&& 1  \tam
\mat 
(1)(2)(3) \\ (12)(3) \\ (13)(2) \\ (1)(23) \\ (123) \tam
\]
where $\A_1 = (1)(1')(22')(33')$, 
$A_{12} = (11'22')(33')$,
and so on.
The corresponding intertwining is:
\[
\mat 
\delta \\ 1&0 \\ 1&&0 \\ &&&\delta \\ &&& 1 &0 \tam
\basss
=
\mat 
\delta^3 & \delta^2 & \delta^2 & \delta^2 & \delta\\
\delta^2 & \delta^2 & \delta & \delta & \delta\\
\delta^2 & \delta & \delta^2 & \delta & \delta\\
\delta^2 & \delta & \delta & \delta^2 & \delta\\
\delta & \delta & \delta & \delta & \delta\\
\tam 
\mat 
\delta&1&1 \\ &0& \\ &&0& \\ &&&\delta&1 \\ &&&  &0 \tam
\]
\[
\mat 
0&1 \\ &1 \\ &&0&&1 \\ &&&0&1 \\ &&&& 1  \tam
\basss
=
\basss
\mat 
0& \\ 1&1 \\ &&0&& \\ &&&0& \\ &&1&1& 1  \tam
\]
- the $\sigma_{ij}$ identities will be clear.
The Smith form of a matrix has the same rank, so
\[
\Gram_3(0) \leadsto
\mat 
\delta (\delta-1)(\delta-2) \\ & \delta(\delta-1 ) \\
&& \delta(\delta-1 ) \\ &&& \delta(\delta-1 ) \\ &&&& \delta \tam
\]
is forced here.



\subsection{Potts functors, symmetry and simple modules} 
$\;$   \label{ss:PottsFunctors}



\noindent 
This is an opportune moment to consider the Potts functors (see e.g. \cite{Martin91,Martin08a}), 
which facilitate the analysis of key representations that we will need
(culminating for example in (\ref{pa:simpleS0})). 

\mdef   \label{de:PottsFunctor}
%
For a given commutative ring $\Fk$, 
$\Mat$ denotes the monoidal category of matrices over $\Fk$ with Kronecker product as monoidal product
(in our convention 
$\Mat(m,n)$ is $m\times n$ matrices, with $m$ rows, and so on). 
For fixed $Q \in \N$,
$\Pcat(Q)$ 
(or 
simply 
$\Pcat$) 
denotes the partition category with $\delta_c = Q$
as in (\ref{de:Pcat}).
For fixed $Q \in \N$ 
the Potts functor is a strict monoidal functor 
$$
\Potts : \Pcat(Q) \rightarrow \Mat   .
$$
The functor may be given as follows.

\mdef  \label{de:Potts}
The category $\Pcat$ 
is linearly-monoidally generated by the unique partition elements in 
$\Pcat(1,0)$ and $\Pcat(0,1)$; the single-part partition elements in $\Pcat(1,2)$ and $\Pcat(2,1)$;
and the elementary permutation element in 
$\Pcat(2,2)$. 

The image under $\Potts$ of the generating object 1 is $Q$.
So the image of an element in $\Pcat(i,j)$ lies in $\Mat(Q^i, Q^j)$.
The rows of a matrix in $\Mat(Q^i, Q^j)$ are indexed by words in 
$\ul{Q}^i  \cong \hom(\ul{i}, \ul{Q})$;
and the columns by words in $\ul{Q}^j$. 
The images for $Q=2$ of the generators above are
\beq   \label{eq:Potts1}
\Potts(\{\{1\}\}) = \mat 1\\1 \tam
\hspace{1.8in}   
\Potts(\{\{1'\}\}) = \mat 1 & 1 \tam
\eq 
\beq 
\Potts(\{\{1,1',2'\}\}) = 
\mat 
 1&0&0&0 \\ 
0&0&0&1\tam
\hspace{1.358in}   
\Potts(\{\{1,2,1'\}\}) = 
\mat 1&0\\
0&0 \\ 
0&0\\
0&1
 \tam
\hspace{.3in} 
\eq 
\beq  \label{eq:Potts12}
\Potts(\{\{1,2'\},\{2,1'\}\}) = 
\mat 
1&0&0&0 \\ 
0&0&1&0 \\ 
0&1&0&0 \\ 
0&0&0&1 \tam
\eq 


More generally, 
for the image $\Potts(p)$ of a partition basis element $p$,
noting that each matrix entry position corresponds to a $Q$-state Potts configuration (or $Q$-colouring) on the corresponding vertex set, 
all vertices in the same part 
of $p$ 
must have the same `colour' for an entry 1, otherwise the entry is 0.
(The ``Potts rule''.)

\mdef Examples. Derived from (\ref{eq:Potts1})-(\ref{eq:Potts12}) above, 
keeping $Q=2$, we have 
(with $T$ for transpose) 
\beq 
\Potts(\{\{1,2,...,n\}\}) \;\; = \; \mat
1&0&0&...&0&0&1   \tam^{\! T} 
\;\; \in \;\; \Mat(Q^n,Q^0)
\eq 
\beq 
\Potts(\{\{1,2,...,n,1'\}\}) \; = \mat
1&0&0&...&0&0&0 \\  0&0&0&...&0&0&1   
\tam^{\!\! T} 
\;\; \in \;\; \Mat(Q^n,Q^1)
\eq 
\ignore{{ 
With the basis ordered as 111,112,121,122,211,212,221,222, we have
\beq 
\Potts(\{\{1\},\{2\},\{3\}\}) = \mat
1&1&1&1&1&1&1&1
\tam ,
\hspace{1cm}
\Potts(\{\{1,2\},\{3\}\}) = \mat
1&1&0&0&0&0&1&1
\tam ,
\eq 
}}

\newcommand{\Pottss}[1]{\Potts_{#1}}  

\mdef  \label{pr:SWdual01}
Observe that the functor 
$\Potts = \Pottss{(Q)} \;$ 
(momentarily writing $\Potts = \Pottss{(Q)}$  to emphasise the $Q$-dependence)
is invariant under the diagonal action of 
the symmetric group
$\Sym_Q$
(the element $\sigma_1 \in \Sym_Q$ takes word $w=11..1$ to 
$\sigma_1 w =22..2$ and so on). 
This is simply the Potts symmetry \cite{Martin91}, 
but in algebraic terms,
for $(\sigma)$ the matrix form of an element in $\Sym_Q$, and $p \in \Pcat(n,m)$ 
\beq  \label{eq:SW00}
\mat \sigma \tam^{\otimes n} \Potts(p) 
\; = \; \Potts(p)  \mat \sigma \tam^{\otimes m}
\eq 

\mdef The examples of (\ref{eq:SW00}) that we will use concern $p\in \rmP_{\ul{n}}$.
For example, with $Q=2$ and for any partition $p$ of $\ul{n}$ we have 
$$
\mat 0&1 \\ 1&0 \tam^{\otimes n} \Potts(p) 
\; = \; \Potts(p)  \mat 0&1 \\ 1&0 \tam^{\otimes 0}
\; = \; \Potts(p) 
$$
and so in particular 
the vector entry $\Potts(p)_{11..1}$ obeys 
$\; \Potts(p)_{11..1} = \Potts(p)_{22..2}$.

\mdef  \label{lem:norbits}
In light of the $\Sym_Q$ invariance, 
the dimension of the image of $\Pcat(n,0)$ under $\Potts$ is thus 
bounded above  
by the number of orbits of the $\Sym_Q$ action on $\ul{Q}^n$.

We will see in (\ref{lem:orb1}) that this number is the 
number of partitions of $\ul{n}$ into at most $Q$ parts,
and 
in (\ref{lem:basis1})
that the bound is saturated.

\medskip

\mdef 
%
{\em Lemma}.   \label{lem:SSAe}
Let $A$ be a $\C$-algebra and $e^2 = e \in A$ with the subspace $eAe=\C e$. 
Suppose $V$ is a space and 
$\rho : A \rightarrow End(V)$
is a semisimple representation of $A$ with
$Trace_{\rho}(e) =1$.
Then 
the irreducible $A$-module decomposition of $V$ contains 
the 
irreducible $head(Ae)$ with multiplicity one; 
and 
the submodule $\rho(Ae)V$ obeys $\rho(Ae)V \cong head(Ae)$
as an $A$-module.

\noindent {\em Proof}. It follows 
from the setup 
that $e$ is a primitive idempotent and so the left ideal  $Ae$ is indecomposable projective, with  $Trace_{Ae} (e) =1$. 
The projection by $e$ kills all simples in $V$ not isomorphic to $head(Ae)$; 
and the trace shows that its multiplicity is one
(a character is the sum of component irreducible characters).
   \qed 

\medskip 

\noindent 
Corollary: 
\label{pa:dimD0}
Observe that 
in $P_n(Q)$ 
the idempotent       
normalisation   of
the partition
$E_0 \; := \; b_0^n$ 
from (\ref{exa:ww})
is $\frac{1}{Q} E_0 \;$ and, 
fixing $Q \in \N$, 
$\; Trace(\Potts(E_0)) =Q$. 
Working over $\C$  
we have that $\Potts(P_n)$ lies in the natural basis choice of 
$End((\C^{Q})^n)$. Observe from the construction that  
$\Potts(P_n)$ is semisimple 
(cf. e.g. \cite{Adamson71}). 
We thus deduce 
from the Lemma 
that the simple head of the left module 
$P_n E_0 \cong \Pcat(n,0) $ 
occurs exactly once as a summand of the Potts representation 
$\Potts(P_n = \Pcat(n,n))$. 
Indeed since $\Potts(P_n E_0) = \Potts(P_n) \Potts(E_0)$ and $\Potts(E_n)$ is primitive,
the dimension of 
simple module 
$\DDD^Q_n(0)$ is the vector space dimension of $\Potts(\Pcat(n,0))$.

\mdef  \label{pa:2bas}
Indeed for $Q=2$ it 
is straightforward to see 
that the $2^{n-1}$ basis elements of $P_n E_0$  corresponding to 
(up to) two-part partitions have independent images in $\Potts$; and that elements of higher order then do not.
Thus the dimension of the simple head is $2^{n-1}$. 

\newcommand{\cruder}{<}

\ignore{{ 
For example consider the three-part partition 
$q= \{\{1\},\{2\},\{3\}\}$.
In general we write $p \cruder p'$ 
to mean that $p$ is a 
cruder partition than $p'$.
Thus 
$ \{\{1,2\},\{3\}\} \cruder  \{\{1\},\{2\},\{3\}\} $  
and so on.

$
\Potts(\{\{1,3\},\{2\}\}) = \mat
1&0&1&0&0&1&0&1
\tam ,
$
$
\Potts(\{\{2,3\},\{1\}\}) = \mat
1&0&0&1&1&0&0&1
\tam .
$
So  
$$
\Potts(q) \; - \; 
\sum_{p < q} 
\left( \Potts(p) -\sum_{p'<p} \Potts(p') \right)
\;=\; \mat 
0&0&0&0&0&0&0&0
\tam 
$$

}}

\medskip

To illustrate this, and the generalisation to 
determine $\dim\Potts(P_n E_0)$ for 
all $Q$ that we will use, 
we first introduce a bit more notation.

\mdef  \label{pa:(123)}
We may  
follow \cite[Ch.8]{Martin91} in writing $(12)(3)$ for the set $\{ \{1,2\},\{3\}\}$, and so on, just to avoid curly-bracket overload.

\newcommand{\sh}{sh}

\mdef  \label{de:shape}
The {\em shape} $\sh(p)$ of a set partition $p$ is the list of sizes of its parts arranged as an integer partition. Thus $\sh: \Part_S \rightarrow \Lambda_{|S|}$. 
Example: $\sh((12)(3)(4)) = (211)\vdash 4$.

\mdef  \label{pa:floor1}
Recall 
from (\ref{pa:Pn}) that 
we 
write $\rmP_{i,j}$ 
for the partition basis of 
the morphism set $\Pcat(i,j)$
(and  
may abbreviate $\rmP_{n,n}$ as $\rmP_n$).
Thus there is a unique element in each 
$\rmP_{i,j}\siz_1$. 
And $\;\rmP_{n,0}\siz_2  \;\cup\;  \rmP_{n,0}\siz_1 $ 
yields the basis of 
the image under $\Potts$  of 
$P_n E_0  \cong \Pcat(n,0)$  in (\ref{pa:2bas}) above
(when $Q=2$),
noting (\ref{de:PnE0basis}).  
\hspace{1cm} 
Example:
\beq  \label{eq:P123}
\rmP_{3,0}\siz_1 = \{   (123)  \}
,\qquad 
\rmP_{3,0}\siz_2 = \{   (12)(3), (13)(2), (1)(23)  \}
,\qquad 
\rmP_{3,0}\siz_3 = \{   (1)(2)(3)  \}
\eq

\mdef 
For any set $S$ a function $f:S\rightarrow T$ to any target defines an equivalence relation, and hence a partition 
$\varpi(f)$
of $S$, by $a\sim b$ if $f(a)=f(b)$.
Thus in particular, fixing $Q \in \N$,  
there is a partition $\varpi(w)$ of $\ul{n}$ for every word 
$w$ in $\ul{Q}^n$
as in (\ref{de:Potts}). 
For example $\varpi(112)=(12)(3)$. 

Observe that 
the image of the map $ \varpi : \ul{Q}^n    \; \rightarrow \;\; \rmP_{\ul{n}} \; $
is $\; \sqcup_{i=1}^Q \; \rmP_{\ul{n}} \siz_i $. That is, 
$$
\varpi : \ul{Q}^n    \; \rightarrow \;\;  \sqcup_{i=1}^Q \; \rmP_{\ul{n}} \siz_i 
$$
is  well-defined and surjective.

\mdef \label{lem:orb1}
Keep $Q \in \N$ fixed. 
Observe here that the set $\varpi^{-1}(\varpi(w))$ is the orbit of 
$w \in \underline{Q}^n$ under the diagonal $\Sym_Q$ action as in (\ref{pr:SWdual01}). 
Thus $| \sqcup_{i=1}^Q \; \rmP_{\ul{n}} \siz_i |$ 
is the number of $\Sym_Q$ orbits. 
For example, when $Q=2$, 
$\varpi^{-1}((12)(3)) = \{ 112, 221 \}$;
and 
when $Q=3$,  
$\varpi^{-1}((1234)) = \{ 1111, 2222, 3333  \}$
and
$\varpi^{-1}((12)(3)(4)) = \{ 1123, 1132, 2213, 2231, 3312, 3321  \}$.

Aside: To
choose representatives for the $\Sym_Q$ orbits  
we could use
dictionary order
for example.
Fixing $Q$ and $p \in \sqcup_{i=1}^Q \; \rmP_{\ul{n}} \siz_i  $,
let us write $\varpi^- (p)$ for the representative in $\varpi^{-1}(p)$. 
For example $\varpi^- ((12)(3))=112$.

\mdef  \label{eg:123}
Considering the partitions of $\ul{3}$ organised as in (\ref{eq:P123}) 
we have (with $Q=2$):
\beq \label{eq:123}
\Potts((123)) = \mat 1 &0& 0& 0& 0&0&0&1  \tam^T 
, \qquad 
\eq 
\beq  
\Potts((12)(3)) = \mat 1 &1& 0& 0& 0&0&1&1  \tam^T 
= \mat 1&0&0&1 \tam^T \otimes \mat 1&1 \tam^T 
\eq 
In fact this is the first instance so far in which the Kronecker convention in $\Mat$ is relevant to how we write the image.
Here we say that the possible states of vertices 123 in the $Q=2$ state case are ordered as 111, 112, 121, 122, 211, 212, 221, 222.
(All other cases will be clear from this.)
This is sometimes called aB convention since in writing $A\otimes B$ the first block we write is given  by $B$ multiplied by the first entry of $A$. 

We also have, still with $Q=2$,
\[
\Potts((13)(2)) = \mat 1 &0& 1& 0& 0&1&0&1  \tam^T 
, \qquad
\] \beq \label{eq:1-23} 
\Potts((1)(23) = \mat 1 &0& 0& 1& 1&0&0&1  \tam^T
= \mat 1&1 \tam^T \otimes \mat 1&0&0&1 \tam^T
\eq 
Clearly these (\ref{eq:123}-\ref{eq:1-23}) are all independent of each other. 
Meanwhile 
$$
\Potts((1)(2)(3)) = \mat 1&1&1&1&1&1&1&1 \tam^T 
$$
is not then also independent. 

\mdef
Examples with $Q=3$. $\;$ 
The states of vertices 123 in case $Q=3$ are
given by
$\; \ul{3}^3 = $ 
\\  $\mbox{ }$  $\;$ \hspace{-.321in} 
$\{ \!\!\!\!$
{\tiny{ 
111, 112, 113, 121, 122, 123, 131, 132, 133, $\;$
211, 212, 213, 221, 222, 223, 231, 232, 233, $\;$ 
311, 312, 313, 321, 322, 323, 331, 332, 333. 
}}  $\!\!\!\!\!\!\}$
\\ 
The orbits of the diagonal action of $\Sym_3$ are \quad 
111, 222, 333, \qquad
112, 113, 221, 223, 331, 332, \qquad 
121, 131, 212, 232, 313, 323, \qquad
122, 133, 211, 233, 311, 322, \qquad
123, 132, 213, 231, 312, 321.
\\
Thus we only need to give the five vector entries in a transversal of this partition to give any $\Potts(p)$ for $ p \in \Part_{\underline{n}}$
(here with $n=3$).  

Consider 
$\ul{3}^5$ (243 words altogether). 
Representatives of $\Sym_3$-orbits 
can be organised according to integer partitions of 5 into at most 3 parts
(thus (5), (41), (32), (311), (221)), 
these being the possible outcomes of $\sh(\varpi(w))$.
A complete list of representatives
are contained in the sets: 
$r_{(5)} = \{ 11111 \}$,  \qquad 
\( r_{(41)} =\{ 11112, 11121, 11211, 12111, 21111 \} \),  \qquad 
\\ 
\( r_{(32)} =\{ 11122, 11212, 11221, 12112, 12121, 12211, 21112, 21121, 21211, 22111 \} \), \qquad 
\\ 
\( r_{(311)} =\{ 11123, 11213, 11231, 12113, 12131, 12311, 21113, 21131, 21311, 23111 \} \), \qquad 
\\ 
\( r_{(221)} =\{ 11223, 11232, 11322, 12123, 12132, 13122, 12213, 12312, 13212, 
12231, 12321, 13221, 31122, 31212, 31221 \} \).
Note, one
representative 
in an orbit of size 3 and 40 in orbits of size 6 --- a 41d space.
Of course these correspond to the 41 set partitions of $\ul{5}$ of size at most 3.

Thus to give $\Potts((1)(2)(3)(45))$ we only need to give $\Potts(p)_w$
for $w$ in the lists above. For example 
$\Potts(p)_{11111} =1$; $\Potts(p)_{11112} =0$, 
$\Potts(p)_{11121} =0$, $\Potts(p)_{11211} =1$ and so on.

\mdef 
Let $S$ be any given set, and $p,q \in \rmP_S$. 
Write $p \prec q$ 
(or $q \succ p$)
to denote that partition $p$ is less refined than $q$.
For example
$(123)\prec (12)(3) \prec (1)(2)(3)$.
Note that this gives a partial order on $\rmP_S$. 

\mdef   \label{lem:basis1}
Observe 
from (\ref{de:Potts})
that $\Potts(\varpi(w))_w =1$ (recall words $w$ index vector elements here, cf. (\ref{pr:SWdual01})). 
And observe that for every 
partition $q\nsucceq \varpi(w)$ we have $\Potts(q)_w =0$.
This verifies the independence of the set of 
images of 
partitions into at most $Q$ parts.

Then 
with (\ref{lem:orb1}) 
- taking account of the diagonal action of $\Sym_Q$ - 
and (\ref{lem:norbits}) 
it also yields that the images of these partitions form a spanning set
of our space $\Potts(\Pcat(n,0) \cong P_n E_0)$.
\\
Specifically in our example (\ref{eg:123}) above we have 
$\varpi^{-1}((123))=\{111,222\}$, 
$\varpi^{-1}((12)(3))=\{112,221\}$, 
$\varpi^{-1}((1)(23))=\{122,211\}$, 
$\varpi^{-1}((13)(2))=\{121,212\}$, 
whose union is $\ul{2}^3$, so $\Potts((1)(2)(3))$ must be {\em dependent}.

\mdef 
Aside:
Observe in the  
`Potts rule' 
(see (\ref{de:Potts}))
that 
different parts may have different
colours, but of course if there are only two colours then if there are more than two parts some of them will have the same colour, and potentially be indistinguishable from cruder partitions. 
We deduce that $\Potts$ has a significant kernel.

Consider in particular the representation of $P_n(Q)$ given by $\Potts$. 
This is a representation on $(\C^Q)^{\otimes n}$, and hence of dimension $Q^n$.
Furthermore the image is a *-algebra and hence a semisimple quotient of $P_n(Q)$.
Indeed the corresponding quotient tower gives a Jones basic construction 
in the sense, for example, of \cite{GoodmanDelaharpeJones89}
(in fact two such).

\newcommand{\ff}{{\mathsf{f}}} 

An example of an element in the kernel is the 
unnormalised Young antisymmetrizer of rank $Q+1$
(in $P_n$ for any $n>Q$ by the natural inclusion). 
We write this for now as 
$\ff_{Q+1} = \sum_{g \in S_n} (-1)^{len(g)} g$. 
For example $\ff_2 = (11')(22') - (12')(21')$.
Evidently this vanishes in $\Potts$ when $Q=1$. 
More generally the image of $\ff_{Q+1}$ vanishes whenever we have symmetry under the exchange of two vertices among $Q+1$ - which holds in $\Potts$ for any element of $P_n$ since we only have $Q$ colours with which to distinguish vertices.

\mdef
We have seen 
that the partition category action commutes with the action of the symmetric group $\Sym_Q$ permuting the `colours' 
(as opposed to the vertices)
- the diagonal action on the corresponding tensor spaces. 
This is the Potts symmetry (of the Potts model). 
Indeed these actions are in (an appropriate generalisation of) Schur--Weyl duality.


\section{
Proof of semisimplicity condition 
for the   ordinary 
$\dl =1$ case
recalled}   $\;$ 
\label{ss:review}

In this paper we follow and extend the original method  
used for the partition algebras themselves.
In this subsection we 
again 
assume familiarity with the partition algebra
(and partition category), 
and hence also the  
group algebra of the symmetric group $\Sym_n$,
and can then give a convenient brief overview of our approach in terms of the partition algebra case.

\medskip




\subsection{Gram matrices, tower representation theory and semisimplicity}
$\;$ 
\label{ss:gram-tower-semi}

\ignore{{ 
\begin{figure} 
\[
\includegraphics[width=12.8cm]{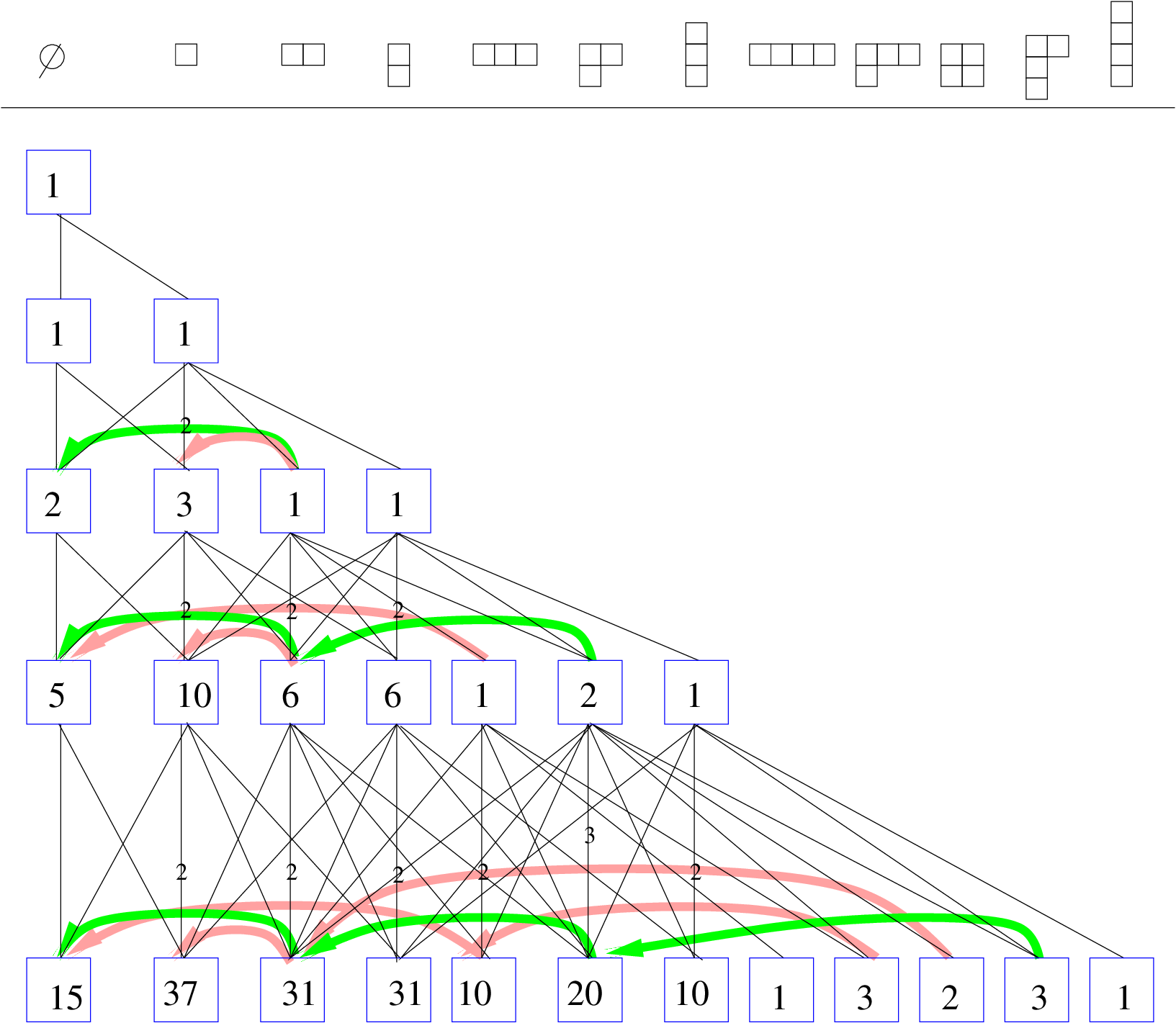}
\]
\caption{Augmented Bratelli diagram for $P_0 \subset P_1 \subset ...\subset P_4$. 
Vertices are standard modules with index as shown at the top of their column; and dimension shown in the box. Black edges indicate restriction rules (with multiplicities) so dimensions can be checked.
\textcolor{green}{Green} arrows indicate morphisms for $\delta=1 \in \C$
(see main text for commentary).
\textcolor{pink}{Pink} arrows indicate  morphisms for $\delta=2$.
(Morphisms for other $\delta$s 
omitted 
to avoid clutter.)
\label{fig:PnBratelli}}
\end{figure}
}}

\ppmy{ 
\mdef 
\ppm{[red since now copied in Intro!:]
Over $\C$ the algebras $P_n$ are generically semisimple, 
and in all cases with $\delta\neq 0$ the standard modules give a basis for 
the Grothendieck group,
so much of their combinatorial representation theory  
is encoded in the standard restriction diagram 
(generically the Bratelli diagram, or else 
a suitably enhanced version thereof)
for the tower of algebras 
$P_n \subset P_{n+1}$
in the generic case.
See Fig.\ref{fig:PnBratelli}.
\\
In Fig.\ref{fig:PnBratelli}  
we also record the morphisms between standard modules 
in the cases $\delta = 1,2$.
The Figure encodes a lot of information. However here we
will only need some specific observations from it. 
[-Delete this now!]
}
}

Like Figure~\ref{fig:PnBratelli}, 
the gram matrices for the standard 
modules also 
contain a lot of information, but again 
we will only need parts of it, 
as in (\ref{de:Pn-spine}),
that are relatively readily accessible. 
To illustrate, 
consider the examples given 
explicitly in (\ref{eg:gramSn0}-\ref{eg:gramS30}). 



\mdef  \label{pa:gram-det-order}
The general $n$ form of the gram determinant 
$\det(\Gram_n(0))$
is somewhat non-trivial. But it is 
easy to see that the exponent of factor $\delta$ is 
the dimension of the module
(which is the Bell number $B(n)$ from (\ref{de:bell1})); 
and the total 
polynomial degree is 
$\sum_{l=1}^n l S(n,l) $ 
(coincidentally also dim$(\SSS_n((1)))$)
--- this is the sum of the exponents down the main diagonal, by the composition rule; 
and it is straightforward to check that this gives the highest order term in the Laplace expansion.

\mdef  \label{pa:simpleS0}
On the other hand we observe that the polynomial 
$\det(\Gram_n(0))$
must,
on representation theory grounds, 
contain factors as follows. 
(NB, we sometimes write $\SSS_n(0)$ for $\SSS_n(\emptyset)$ - purely for aesthetic reasons.)
The head of $\SSS_n(\lambda)$ 
when working over a field
depends on the specialisation $\delta\leadsto\delta_c$. For given $\delta_c$ we write $\DDD_n^{\delta_c}(\lambda)$ for the head,
as in (\ref{lem:ss01}).
For the factors $(\delta -1)$ in the determinant we observe that 
dimension of head of $\SSS_n(0)$ 
when $\delta_c =1$
is  $\dim(\DDD_n^{1}(0)) = 1$. 
We can observe this directly in our low rank 
diagram in Fig.\ref{fig:PnBratelli};
or in general by inspecting the Potts representation.
Thus the rank of the 
morphism (unique up to scalars, since necessarily head to socle) from 
standard module $\SSS_n(0)$ to 
costandard module $\nabla_n(0)$   
is 1. 
The gram matrix $\Gram_n(0)$ intertwines these two modules over any ground ring.
By (\ref{eq:EPE}) it
always maps the head to the socle, 
but its rank can
vary in specific specialisations (evalutations of $\delta$).
In order to have the right rank when $\delta=1$, 
it follows that all but one entry in the Smith form 
of $\Gram_n(0)$ over $\Z[\delta]$ (strictly speaking for a Smith form it should be over, say,
$\C[\delta]$, but here it makes no difference)
has a factor $(\delta -1)$.
So there are at least dim$(\SSS_n(0))$-1 such factors
in the determinant. 
That is, writing 
\beq  \label{eq:gram1} 
\det(\Gram_n(0)) = \prod_{\delta_c} (\delta-\delta_c)^{\alpha_{\delta_c}} 
\eq 
we have 
a bound on the exponent:
$\alpha_1 \geq \dim(\SSS_n(0))-1$.
A similar argument gives a bound 
\beq \label{eq:bound1}
\alpha_Q \geq \dim(\SSS_n(0))-\dim(\DDD_n^Q(0))
\eq 
on the number of
factors of each form $(\delta-Q)$ for $Q \in \N$. 
In general we have
\beq   \label{eq:dimD1} 
\dim(\DDD_n^Q(0))  \;=\; 
 \sum_{l=1}^{Q} S(n,l)
\eq 
arguing as in (\ref{pa:dimD0} - \ref{lem:basis1}). 
Hence
\[
\dim(\SSS_n(0))-\dim(\DDD_n^Q(0))  \;=\; 
\sum_{l=1}^n S(n,l) - \sum_{l=1}^{Q} S(n,l)
\; = \; \sum_{l=Q+1}^{n} S(n,l)
\]
Summing we obtain an overall bound for the polynomial degree
\[
degree(|\Gram_n (0) | )    \; \geq \; 
\sum_Q \sum_{l=Q+1}^{n} S(n,l) 
\; = \; \sum_{l=1}^n l S(n,l)
\]
By 
our earlier observation in (\ref{pa:gram-det-order})
this is saturated.
Thus the formulae (\ref{eq:bound1}) are equalities; and 
all factors in (\ref{eq:gram1}) come from 
these $Q$ values. Thus this gram determinant does not vanish for any other $\delta$ value. 
Thus $\SSS_n(0)$ is simple for all $n$ for all other $\delta$ values. 

\medskip

On the other hand we have the following. 

\mdef \label{pa:indicator}
\ignore{{ 
Finally we 
argue
by Frobenius reciprocity 
using properties of the standard module restriction rules (as indicated in Fig.\ref{fig:PnBratelli} for example), 
that 
}}
If $\delta_c \neq 0$ and 
$\SSS_n(0)$ is simple for all $n$ then 
$P_n (\delta_c)$ is semisimple for all $n$.

\medskip 

\newcommand{\ind}{Ind\;}
\newcommand{\res}{Res\;}
\newcommand{\mul}{\lambda}
\newcommand{\nul}{\nu}
\newcommand{\bigplus}{\sum} 
\newcommand{\lleadsto}{\mbox{\reflectbox{$\leadsto$}}}

We 
now 
briefly recall the argument for this from \cite{MartinSaleur94b}, in preparation
for the 
modifications needed for the 
2-tonal case (which is in \ref{frecip}).
%

\noindent 
(I)
Suppose $n$ fixed. 
For brevity we assume the construction and basic properties of
the set $\SSS$ of standard $P_n$-modules $\SSS_n(\lambda)$ with 
$\lambda$ an integer partition of $l\leq n$. 
We write $|\lambda| = l$, as in (\ref{pa:globloc101}).

If a module $M$ has a filtration by a set ${\mathcal T}$ of modules we say that it has 
a ${\mathcal T}$-filtration.

\noindent
(II)
Note that if $P_n$ is not semisimple then there must be a homomorphism
$\psi: \SSS_n(\mul) \rightarrow \SSS_n(\nul)$ for some 
distinct $\mul,\nul$;
in particular with $|\mul| > |\nul |$
(the latter condition comes from localisation functors and Maschke's Theorem,
considering (\ref{pa:globloc101})
and in particular (\ref{eq:uppertriangular});
or equivalently from quasiheredity).
And hence in particular some $\SSS_n(\nul) $ is not simple. 
So it is enough to prove that all $\SSS_n(\nul) $ are simple. 
We do this by induction on $|\nul|$.
\ignore{ 
Assuming $\delta \not\in \N_0$ then it 
is straightforward to eliminate cases with $|\mul| = |\nul | +1$,
so we consider $|\mul| > |\nul | +1$.
}

\noindent(III) 
We will use the reciprocity 
\beq \label{eq:FR}
\Hom(\Ind \SSS_n(\mul), \SSS_{n+1}(\nul) ) \cong \Hom(\SSS_n(\mul) , \res \SSS_{n+1}(\nul) )
\eq 
where the $\Ind$ and $\res$ functors are those associated to 
the homomorphism $ P_{n} \hookrightarrow P_{n+1}$ given by 
$p \mapsto p\otimes 1$. 
By the usual $P_{n} \hookrightarrow P_{n+1}$ 
Induction/Restriction rules 
(as indicated,
at least at the level of characters, 
in Fig.\ref{fig:PnBratelli} for example), 
we have 
\beq \label{eq:Indres}
\Ind \SSS_n(\mul) \;\lleadsto  \;
\{ \SSS_{n+1}(\rho) \; | \;    \rho \in \mul\pm \square   \}
,
\qquad \hspace{.21in} 
\res \SSS_{n+1}(\mul) \;\lleadsto  \; 
\{ \SSS_{n}(\rho) \; | \;    \rho \in \mul\pm \square  \}
,
\eq 
where $\lleadsto {\mathcal T}$ indicates 
${\mathcal T}$-filtration; 
and $\lambda\pm\square $ indicates 
a set of factors 
adjacent to $\lambda$ in the Young graph (add or remove a box from the Young diagram, or add then remove), and hence 
with weights of rank between $|\lambda|+1$ and $|\lambda|-1$.
A schematic for this is:
\beq  \label{eq:FR111}
\includegraphics[width=10cm]{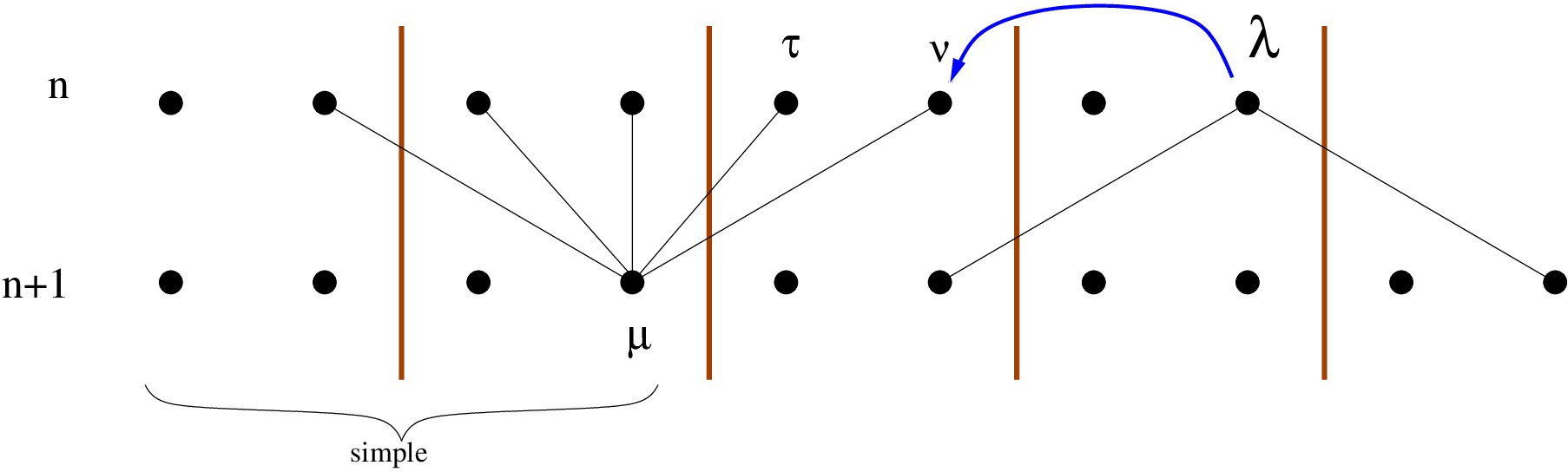}
\eq
where the brown walls 
represent
partitioning of 
the index set by increasing rank;
and the black lines indicate the factors of 
$\res\SSS_{n+1}(\mu)$ and $\Ind\SSS_n ( \lambda)$ respectively. 

\noindent
(IV)
So now suppose for an induction that modules 
$\SSS_n(\mu)$
of ranks $|\mu|$ up to $l$ are simple 
(for all $n$ - hence for fixed but arbitrary $n$).
The case we have 
concluded in (\ref{pa:simpleS0}) above - rank 0 - will 
eventually be the base of induction,
but for this proposition the base is exactly the hypothesis.
Excluding $\delta=0$ this implies no maps $\psi$ in
to $\SSS_n(\mu)$ (apart from self-isomorphism), 
so then in particular
for every $\mu$ of rank $|\mu| =l$ we have 
$\; \Hom(\Ind \SSS_n(\mul), \SSS_{n+1}(\mu) ) =0 \;$ 
for all $\mul$ of rank at least $|\mu|+2$, 
by 
the $\SSS$-module restriction rule from (\ref{eq:Indres}),
the upper-triangular simple decomposition matrix property
(\ref{eq:uppertriangular}), and Schur's Lemma.
But suppose (for a contradiction, in the inductive step to $l+1$) some 
$\SSS_n(\nu)$ with 
$\nu$ of rank $l+1$ 
is not simple, and hence has a map in, from $\SSS_n(\lambda)$ say. 
See the figure in (\ref{eq:FR111}).
The factors in any  
{$\res \SSS_{n+1}(\mu)$} as above are either in their own block by 
the inductive assumption, or of equal rank to $\nu$ and hence
(cf. (\ref{pa:globloc101}))
not involved
in extensions with $\SSS_n(\nu)$,
and for a suitable $\mu$ they include $\SSS_n(\nu)$. 
Thus there is a map in 
$  \Hom(\SSS_n(\mul) , \res \SSS_{n+1}(\mu) )  $. 
This then contradicts (\ref{eq:FR}). 
We deduce that every $\SSS_n(\nu)$ with 
$\nu$ of rank $l+1$ is simple, which completes the inductive step.

\medskip

\mdef 
Given (\ref{pa:indicator}) and (\ref{pa:simpleS0}) we have that $P_n$ is semisimple over $\C$
unless $\delta \in \N_0$. 


\noindent 
(Cf. in particular the formulation of our Theorem~\ref{th:main} here.)




\medskip

\section{General preparations for the 2-tonal ($d=2$) case}
\label{ss:maintheo}

In  section~\ref{ss:2tonal} we will prove our main 
result, 
Theorem~\ref{th:main}.
In \S\ref{ss:tonal-prep} we 
appropriately 
generalise the Potts functor machinery of \S\ref{ss:PottsFunctors}.
In \S\ref{ss:gramgen00} we assemble the gram matrix machinery that we will use,
generalising \S\ref{ss:gram-tower-semi}. 


\subsection{Potts functor preparations for the tonal cases}  \label{ss:tonal-prep}  $\;$

\mdef
For later use it is convenient to  
recall that even partitions span a 
monoidal subcategory of $\Pcat$. 
This category is monoidally generated by 
$\{\{1,2\}\} \in \Pcat(2,0)$,  
$\;  \{\{1',2'\}\} \in \Pcat(0,2)$,
$\; \{\{1,2'\},\{1',2\}\} \in \Pcat(2,2)$,
and
$\{\{1,2,1',2'\}\} \in \Pcat(2,2)$.
$\;$ \confer{\S\ref{ss:PottsFunctors}}

The $\Potts$-images of these generating elements for $Q=2$ are
\[
\Potts( \{\{1,2\}\} ) = \mat  1 \\ 0 \\ 0 \\ 1 \tam ,
\hspace{.241in} 
\Potts( \{\{1',2'\}\} ) = \mat 1&0&0&1 \tam ,
\hspace{.241in} 
\Potts( \{\{1,2,1',2'\}\} ) = \mat 1&0&0&0 \\ &0 \\ &&0 \\ 0&0&0&1 \tam 
\]
and as in (\ref{eq:Potts12}).

\mdef  \label{pa:12-34}
Still with $Q=2$, 
for example
\[
\Potts((12)(34)) = \mat 1&0&0&1 \tam^T\otimes\mat 1&0&0&1 \tam^T  
\hspace{1in} 
\] \[ \hspace{.6in} 
 = \mat 1&0&0&1&0&0&0&0&0&0&0&0&1&0&0&1 \tam^T 
\]
(Recall the basis here is 
{\tiny{1111, 1112, 1121, 1122, 1211, 1212, 1221, 1222, 2111, 2112, 2121,2122,2211,2212, 2221, 2222}}.)

\mdef 
Observe that these matrices have the symmetry not only of commuting with the 
categorified diagonal $\Sym_Q$ matrix action \cite{Martin08a}, but also with the categorified diagonal action of the group 
of signed permutation matrices
(for the non-categorified version see e.g. \cite{benkart2016schur,benkart2017walks}). 
This group, and hence action, is generated by the $\Sym_Q$ matrices  
together with the signed unit matrices.

\mdef 
For example, for $Q=2$ it is sufficient to check the action of matrix $diag(1,-1)$:  
\[
\left( \mat 1&0 \\ 0& -1 \tam^{\otimes 2}  \right) 
\mat 1 \\ 0 \\ 0 \\ 1 \tam  =  
\left( \mat 1&0 \\ 0& -1 \tam  \otimes \mat 1&0 \\ 0& -1 \tam  \right) 
\mat 1 \\ 0 \\ 0 \\ 1 \tam  =  
\mat 1 \\ & -1 \\ && -1 \\ &&&& 1 \tam 
\mat 1 \\ 0 \\ 0 \\ 1 \tam
\] \[ 
=\mat 1 \\ 0 \\ 0 \\ 1 \tam (1) 
=\mat 1 \\ 0 \\ 0 \\ 1 \tam
\mat 1&0 \\ 0& -1 \tam^{\otimes 0}
\]

Observe then that in addition to vectors in $\Potts(\rmP_{n0})$ being eigenvectors 
(with unit eigenvalue) of the diagonal $\Sym_Q$ action, 
{the vectors in   $\Potts(\rmP^2_{n0})$  }
are also 
eigenvectors 
for the diagonal action of the signed permutation matrices, and hence in particular for the diagonal action of the signed unit matrices.
(Hence they intertwine a single irreducible representation - the trivial representation - and so form a simple representation themselves of $P^2_n$, by Schur's Lemma.)

For this diagonal action in rank $n$, consider the example
\[
\mat 1&0 \\ 0& -1 \tam^{\otimes 4} 
 = diag( \; 1,-1,-1,\; 1,-1,\; 1,\; 1,-1,-1, \; 1, \; 1,-1, \; 1,-1,-1, \;1) 
\]
cf. (\ref{pa:12-34}).
The sign here depends on the number of 2's in the word.
And in general on the number of $i$'s for all $i$. 
For invariance (eigenvalue 1) we see that a vector entry in $\Potts(p)$ should be zero 
unless the word has all multiplicities being even.
(The generalisation to any $d$ will be clear.) 

\mdef  
Since we have already established 
in (\ref{lem:basis1})
that the images of partitions of up to $Q$ parts are independent, we deduce that 
(for $n$ even) 
the images of even partitions of up to $Q$ parts give a basis of 
the image of 
the module
$P^2_n E_0 \cong \Pcat^2(n,0)$.
(In the notation of \cite{amm2019tonal}, 
$\cell_\nk(\emptyset, \emptyset) =    
P^2_n E_0 $.)
Observe from the semisimplicity of $\Potts(P^2_n)$ and Lemma~\ref{pa:dimD0} 
(or Schur's Lemma) that the image is isomorphic to the simple head. 


Recall 
from \ref{de:bell1} that 
$\T(m,l) := 
       \left|    
             \rmPe_{2m,0}\siz_l \right|  \;$
- i.e. the number of even partitions into $l$ parts of $\underline{2m}$.

Hence we have established the following Lemma, paralleling (\ref{eq:dimD1}).

\begin{lem}\label{dimL}
Let $\delta =l \in\N$, 
and $n$ even. Then
$$
\dim head(\cell_\nk(\emptyset, \emptyset)) 
\; =  \;
\sum_{t=0}^{l} \T(\frac{\nk}{2}, t)
$$
{(the sum 
runs from $t=0$ to include the  
$n=0$ case).}
\qed
\end{lem}

\ignore{{ 
\ppm{[DELETE:]}
NB: If $\delta =0$ then 
\ppm{formally??}
$\dim L(\emptyset, \emptyset) =0$. 
($P^2_\kn(\delta)$ is
cellular but not quasi-hereditary in this case. $L$ is defined to be
the cell module over its radical.
\ppm{[-this obviously does not work as written. but easy to fix.]})
}}


\section{The 2-tonal ($d=2$) case} $\;$ 
\label{ss:2tonal}


In \S\ref{ss:2tonal}-\ref{ss:avengers-end-game} 
we focus on the tonal case $\dell=2$. 
We show that the
2-tonal partition algebra 
$P^2_n(\delta)$ 
is semisimple over $\C$ 
for all $n$, 
for all $\delta\in\C$ 
except  
$\delta \in \N$.
{(Recall that 
in this paper we have the `statistical physics perspective', focusing on results that hold for large $\nk$, i.e. in the thermodynamic limit \cite{Martin91},
and 
formulations that 
hence 
hold
in practice for all $n$ 
(since partition algebras have global limits
in the sense of \ref{de:GA}),
like absence of non-split extensions. 
It is a straightforward exercise using our results to determine the specific integral values of $\delta$ where non-semisimplicity first arises for given $n$.)}

We  
prove the result 
by first showing that we can reduce this problem to establishing that
the ``spine'' module $\spine = P^2_n E_0$
(when $n$ is even)  
is simple for all
even $n$;
with a similar result for odd $n$, see Lemma~\ref{frecip}. 
We then show this module is simple 
for non-integral $\delta$ 
by finding the degree of the Gram
determinant in $\delta$ and then showing that this degree is saturated by the
integral roots.
(This method is used in \cite{MartinSaleur94b}
for the ordinary case,
and has been reviewed here in \S\ref{ss:review}.)

\mdef   \label{pa:Pn2mod}
Observe (e.g. from \cite[Th.9.7]{amm2019tonal}) that 
$P^2_n$ is generically semisimple over $\C$, and 
the modular properties of $P_n$ as in (\ref{lem:ss01}) carry over to $P^2_n$. 
The index set for simples is  
recalled 
in (\ref{pa:Trep})
below.
A preferred choice of integral lattice forms, given via (\ref{pa:Trep}),
again yields {\em standard} modules for every specialisation.




\newcommand{\WW}{W^2}
\newcommand{\WWW}{W_b^2}
\newcommand{\oWW}{\overline{W}^2}
\newcommand{\oWWW}{\overline{W}_b^2}

\mdef  \label{pa:globloc2}
We will  
use a version of (\ref{pa:globloc101}) for $P^2_n$. $\;$
For given $n$ note that $\WW$ and  $\WWW$ from (\ref{pa:U}) are in $P^2_n$ (these require $n\geq 2$ and $n\geq 3$ respectively).
Note that $\WWW$ is idempotent, and 
if $\delta$ is invertible $\WW$ is also preidempotent. 
\ppmx{[fixed now?: -this seems false at the moment! shall we define $\W$ as $\W_n$ in the cases $n>2$? ...Could define it as $\U$ in case $n=2$, and point out that this is not idempotent! - although this seems ugly!]}
Recall then for example from \cite{amm2019tonal} that 
$\WW P^2_n \WW \cong P^2_{n-2}$;
and that the quotient algebra
$$
A^2_n = P^2_n / P^2_n \WW P^2_n
$$
is semisimple over $\C$ - see in particular 
\cite[(9.3)]{amm2019tonal}. 
\\
In  \cite[(9.3)]{amm2019tonal}   
we see also that the same identities hold replacing $\WW$ with 
$\WWW$, or $\oWW$, or $\oWWW$. 
\caa{For example, the following figure (also in 
(\ref{eq:wpw0})) 
schematically indicates,
here for $n=7$, 
the isomorphism between $W^2_b \; P^2_7(\delta) \; W^2_b \;$  and $P^2_5(\delta)$: 
}
\beq \label{eq:wpw} 
\rb{-.22in}{\ig[width= 8.10cm]{wwpwn.eps}}
\eq

\mdef    \label{pa:Trep}
Next, cf. (\ref{pa:globloc102}),
we recall the   corresponding 
index set {for simple modules} 
of $P_\nk^2(\delta)$,  
denoted $ \Lambda(P_\nk^2(\delta)) $. 

For $(l,m) \in \N_0^2 \;$
let 
$\Lambda(\C(\Sym_l\times \Sym_m)) $ denote the index set for simple $\C(\Sym_l\times \Sym_m)$-modules by pairs of integer  partitions 
$(
\lambda, \mu)$, 
where $\lambda$ is a partition of $l$ and $\mu$ is a partition $m$.
{We write $ ( \lambda, \mu) \vdash (l,m) $  for this.
The `rank' is 
$|  ( \lambda, \mu) | \; := \; l+2m$}.

Then  
(from \cite[(9.3)]{amm2019tonal})
an index set for simple $A^2_n$-modules over $\C$ 
is given by 
\ignore{{ 
\beq \label{eq:LAn} 
\Lambda(A^2_n)\; =\; 
\bigcup_{\ppm{(\lambda,\mu) \; s.t. [???]} \; |  ( \lambda, \mu) |=n} 
\Lambda(\C(S_{|\lambda|}\times S_{|\mu|}))
\eq 
\ppm{[-this seems wrong. and very confusing. delete?]} 
}}
\begin{equation}
     \label{eq:LAnn}
\Lambda(A^2_n)\; =\; 
\bigcup_{(l,m) \in \N_0^2 \; s.t. \; l+2m=n} 
\Lambda(\C(\Sym_l\times \Sym_m))
\end{equation} 
In other words the elements are pairs $(\lambda,\mu)$ 
of rank $|(\lambda,\mu)| = n$.

If $\delta\neq0$ then 
from (\ref{pa:globloc2}) 
and (\ref{de:GFfunctors}) 
(cf. ~\cite[(5.9)]{amm2019tonal}) 
we have 
\begin{equation}
  \label{eq:Index} 
\Lambda(P_\nk^2(\delta)) \; =\; 
   \Lambda(A^2_n)\cup \Lambda(A^2_{n-2})\cup \dots \cup \Lambda(A^2_{0/1 }) \; 
\end{equation}
where $0/1=n\bmod{2}.$

Let 
\[
\gamma^{2,n}=\{(l,m)\in \N_0^2\;:\; (n-(l+2m))/2\in\N_0\}
.
\]
And then from (\ref{eq:LAnn})
we have 
\[
\Lambda(P_\nk^2(\delta)) \; =\;
\left\{ \; 
( \lambda, \mu) \vdash (l,m) \;\; | \; (l,m)\in\gamma^{2,n} \right\}    .
\]


\ignore{{We recall from \cite{amm2019tonal}[Theorem 9.7] that the 2-tonal partition algebra $P_\nk^2(\delta)$ is generically semisimple over $\C$. Hence, similar to the ordinary partition algebra there is a modular system for each $\delta\in\C$, this is discussed in \cite{amm2019tonal}[9.8].}} 
For a fixed $n$   we write 
$\cell_n(
\lambda, \mu)$ 
for the standard $P_\nk^2(\delta)$-module labelled by 
$(
\lambda, \mu)\in \Lambda(P_\nk^2(\delta))$, 
as constructed 
by analogy with  \S\ref{ss:brt} / 
(\ref{pa:globloc102})
above 
(and explicitly for example in \cite[\S7.3]{amm2019tonal}). 
If $|(\lambda,\mu)|=n$ then $\SSS_n(\lambda,\mu)$ is the corresponding $A^2_n$ Specht module induced along $\psi: P^2_n \twoheadrightarrow A^2_n$.
Otherwise
if $n>2$ then we can take 
$\; \SSS_{n+1}(\lambda,\mu)  \; =\;  G_{W^2_b} \SSS_{n-1}(\lambda,\mu)$
(recall (\ref{de:GFfunctors})).
The module this excludes is  $\SSS_2(\emptyset,\emptyset)$,
but see (\ref{de:Pn2-spine}) for this.

\mdef  \label{pa:P2head}
When $\delta\neq0$, {or $n$ is odd}, the set 
$\; \left\{ \; head(\cell_n(
\lambda, \mu)) \;:\;(\lambda, \mu)\in \Lambda(P_\nk^2(\delta))
\right\} \;$ 
forms a complete
set of pairwise non-isomorphic simple $P_\nk^2(\delta)$-modules 
(for a proof see e.g. \cite[Lemma 7.14]{amm2019tonal}). 

Note that a version of the upper-triangular decomposition property 
from (\ref{de:GA})
holds
(see e.g. \cite[9.4]{amm2019tonal}).

\newcommand{\oo}{\emptyset,\emptyset}

\mdef  \label{de:Pn2-spine}
In direct analogy to (\ref{de:Pn-spine}) we have 
$ \cell_\nk(\oo ) = P^2_n E_0$
for all $\delta$ and all even $n$;
and 
$ \cell_\nk((1),\emptyset) = P^2_n E_1$
for all $\delta$ and all odd $n$.

\ignore{{ 
Recall 
$ \; 
\T(m,t) = \mbox{number of ways of partitioning 
a set of size $2m$ into $t$ blocks of even size}
$
\footnote{{ 
Remark: Like Stirling numbers, the numbers $\T(m,t)$   obey 
various nice identities (not needed here) such as 
$$
\T(m,t)=\frac{1}{t!2^{t-1}}\sum_{j=1}^t(-1)^{t-j}{2t \choose t-j} j^{2m}.
$$
They are $A156289$ in the Encyclopedia of Integer Sequences~\cite{SloaneOEIS}.
 }}
 }}



\begin{lem}   \label{lem:dimS}
For an even integer $\nk$ we have 
(cf. \ref{dimL})
$$
\dim \cell_\nk(\emptyset,\emptyset) = 
\sum_{t=1}^{\nk/2} \T(\frac{\nk}{2}, t)
$$  
\end{lem}
\begin{proof}
    This is the dimension of $P^2_n E_0$ by the construction of the algebra
    as in \S\ref{ss:constructPd}, 
    noting (\ref{de:bell1}).
\end{proof}


\newcommand{\Ress}[2]{\Res^{#1}_{#2}}  

\mdef \label{res:Tonal}
Let 
$\; \Ress{n}{n-1}: 
P_\nk^2(\delta)-\text{mod} \;\to\; P_{n-1}^2(\delta)-\text{mod} \;\;$ 
be the restriction functor associated with the inclusion $P_{n-1}^2(\delta)\hookrightarrow  P_\nk^2(\delta)$ given by 
$d\mapsto 1\otimes d$.
Then 
$\Ress{n}{n-1}(\cell_n(
\lambda, \mu))$ 
has filtration by standard modules 
{(see e.g. \cite[11.3]{amm2019tonal})}. 
Below is a schematic illustrating 
{filtration factors of 
$  \Ress{n}{n-1} ( \cell_n( \lambda, \mu))$}
in general position (see \cite[11.3 and 11.4]{amm2019tonal}):


{\small
\begin{equation}
     \label{eq:resschem}
\begin{tikzcd}
	{(\lambda+\square, \mu-\square) } && {(\lambda+\square,\mu) } && {(\lambda+\square,\mu+\square) } \\
	\\
	{(\lambda,\mu-\square) } && {(\lambda,\mu)} && {(\lambda,\mu+\square) } \\
	\\
	{(\lambda-\square,\mu-\square) } && {(\lambda-\square,\mu) } && {(\lambda-\square,\mu+\square) }
	\arrow[from=3-3, to=1-1]
	\arrow[from=3-3, to=1-3]
	\arrow[from=3-3, to=5-3]
	\arrow[from=3-3, to=5-5]
\end{tikzcd}
\end{equation}
}

And here is the same data expressed in a layout matching (\ref{eq:FR111}):

{\small
\[\begin{tikzcd}
	{n-1} && { (\lambda-\square,\mu) } & {(\lambda+\square, \mu-\square)} && {(\lambda-\square,\mu+\square)} & {(\lambda+\square,\mu) } \\
	\\
	n &&&& {(\lambda,\mu)}
	\arrow[from=3-5, to=1-3]
	\arrow[from=3-5, to=1-4]
	\arrow[from=3-5, to=1-6]
	\arrow[from=3-5, to=1-7]
\end{tikzcd}\]
}

\noindent
See Fig.\ref{fig:bratvtwo} for the realisation in low ranks, paralleling Fig.\ref{fig:PnBratelli}.


\medskip 

\begin{lem}\label{frecip}
For $\delta\in \C^\times$, the 2-tonal partition algebras 
obey

\noindent
a)  $P^2_{2\nk}(\delta)$ is semisimple for all $\nk$ if and only if
$\cell_{2\nk}({\emptyset},\emptyset)$ 
is simple for all $\nk$.

\noindent
b)  $P^2_{2\nk+1}(\delta)$ is semisimple for all $\nk$ if and only if
$\cell_{2\nk+1}({(1)},\emptyset)$ 
is simple for all $\nk$.
\end{lem}

\begin{proof} 
{ The proof in each case is a direct analogue of the one  in (\ref{pa:indicator}) for the ordinary case.
In particular consider the `if' direction.

\noindent 
Step (I): 
We
note from {\ref{pa:Trep} }
(coming in turn from \cite{amm2019tonal}, for example)
that the 
simple 
index set for $P_n^2$ is much larger and more complex
than for $P_n$,
but again has a rank $|-|$ correlated to the `layer' in the idempotent reciprocity property
from (\ref{pa:globloc2}). Paralleling Fig.\ref{fig:PnBratelli}  
let us first take a look at the corresponding figure for $P^2_n$ --- 
\ppmm{Fig.\ref{fig:bratvtwo}.}
Observe that the index ranks are even for even $n$ and odd otherwise.
\ppmm{
Observe that the inclusion 
$P^2_{n-2} \hookrightarrow P^2_n$
given by the iterated inclusion 
$P^2_{n-2} \hookrightarrow P^2_{n-1} \hookrightarrow P^2_n \; $}
\ppmm{formally 
preserves index parity under restriction. 
Thus the iterated restriction obeys a parity-preserving form of 
(\ref{eq:Indres}).} 

\ignore{{ 
\begin{figure}[!htbp]
\[
\includegraphics[width=5.5in]{bravtwo.eps}
\]
\caption{Enhanced Bratelli diagram for $P^2_n(\delta)$ up to $n=4$ (cf. Fig.\ref{fig:PnBratelli}).
Observe that the index ranks are even for even $n$ and odd otherwise.
At each level the 
coloured 
arrows represent the maps existing between the standard modules for the specified $\delta$. For example, {on level $\nk=4$ when $\delta=1$ there is a map from 
$(\emptyset, (2))$ to ($\emptyset$,$\emptyset$), indicated by a green arrow, with image of dimension 3.}
\label{fig:bratvtwo}}
\end{figure} 
}}
%

\newcommand{\boxx}{\square}

\noindent
Step (II) then lifts directly
from (\ref{pa:indicator})(II). 
So for $\delta_c \in \C^\times$ it is enough to prove all $\SSS_n(\lambda,\mu)$ are simple. 

\noindent 
Step (III):
Firstly note that 
we have the restriction rules from \ref{res:Tonal}
(coming in turn from \cite[\S11.1]{amm2019tonal}), 
so that the formulation as in 
(\ref{pa:indicator})(III) lifts essentially directly
\ppmm{to each case, 
(a) and (b), 
by considering $P^2_{n-2} \hookrightarrow P^2_n$ and restricting to fixed index parity.}

\ppmm{ 
For example, $\SSS_n(\emptyset,\boxx) \;$ - as in Figure~\ref{fig:bratvtwo} - 
with rank $| \emptyset,\boxx | = 2$,
restricts firstly (along $P^2_{n-1} \hookrightarrow P^2_n$) 
to a module with 
factors $\SSS_{n-1}(\boxx,\emptyset)$ 
(rank $| \boxx,\emptyset  | = 1$) and 
$\SSS_{n-1}(\boxx,\boxx)$
(rank $| \boxx,\boxx  | = 3$).
These factors then restrict again, 
along   $P^2_{n-2} \hookrightarrow P^2_{n-1}$
- the first gives factors 
$\SSS_{n-2}(\emptyset,\emptyset)$
(rank 0),
$\SSS_{n-2}(\emptyset,\boxx)$,
$\SSS_{n-2}(\boxx\!\boxx , \emptyset)$,
and 
$\SSS_{n-2}( 
{{\scalebox{.55}{\rb{-.07in}{$\young(\,,\,)$}}}}
, \emptyset)$
(all rank 2).
Then restricting  $\SSS_{n-1}(\boxx,\boxx)$ we get 
modules with ranks 2 and 4. 
Altogether we have ranks 0,2,4.
(Here we do not track details of the possible order of filtration factors. We will address this shortly.)
And in general starting with rank $r$ we will get 
at most ranks $r-2,r,r+2$. 
 }

\ppmm{Note so far that the restriction part of the schematic in (\ref{eq:FR111})
can be used again here, except that the layers are now 2 apart; and the rank partition looks only at ranks of a corresponding fixed parity.}

\newcommand{\Indd}[2]{\Ind^{#1}_{#2}}

\ppmm{For the induction rules, in the ordinary 
$P_{n-1} \hookrightarrow P_n$ 
case we have,
expressed informally:
$ Ind_{n-1}^{n} - \cong Res^{n+1}_{n}(G_{n}^{n+1} G_{n-1}^{n} -)$ 
(and $G_{n-1}^{n} \SSS_{n-1}(\lambda) = \SSS_n(\lambda)$),
where $G = G_\A$ is as in \ref{de:GFfunctors} and \ref{de:GA}. 
This has a parallel for 
$P^2_{n-1} \hookrightarrow P^2_{n}$ 
as in 
\cite[(11.7)]{amm2019tonal},
which we denote 
$\; \Indd{n}{n-1} : P^2_{n-1}-\!\!\!\!\mod \;\rightarrow\; P^2_{n}-\!\!\!\!\mod$
(using $G_{W^2}$ once instead of $G_\A$ twice). }
Applying this to some $\SSS_{n-1}(\lambda,\mu)$ we thus get a $P^2_n$-module with an $\SSS$-filtration. 
Next applying $\Indd{n+1}{n}-$, we first apply $G-$ to the filtered module (then finally $\Ress{n+2}{n+1} -$). The functor $G-$ is right exact but not exact. 
However,
since it is well defined in all cases including semisimple, 
it behaves `as expected' on characters, 
and this will be sufficient for our purposes. 
\ppmx{ 
[So it remains to derive the 
corresponding property in our double-step case.]
[DONE now?]}

\medskip 

\noindent 
Step (IV):
The base of induction here is,
in each case,
the if-hypothesis in the Lemma. 
Otherwise the argument lifts essentially directly, using (\ref{pa:globloc2}) instead of (\ref{pa:globloc101}).
\ppmm{
The inductive assumption is that all modules $\SSS_n(\mu,\mu')$ with ranks $|(\mu,\mu')|$ up to $l$ are simple. 
In particular the double-scale interpretation of the schematic 
(\ref{eq:FR111}) 
sets up the argument for why 
$\; \Hom(\Ind_n^{n+2}(\SSS_{n} (\lambda,\lambda')) , 
\; \SSS_{n+2}(\mu,\mu')) = 0 \;$
--- here we write $\Ind_{n}^{n+2} \!-\; $ 
for the two-step induction functor;
we write 
$(\lambda,\lambda')$ for the double-index version of $\lambda$ in (\ref{eq:FR111}), hence with rank $l+4$ or greater; 
and $(\mu,\mu')$ for the double-index version of $\mu$. 
The argument is as follows. 
Firstly   $ \SSS_{n+2}(\mu,\mu') $ is simple by assumption,
so it is enough to show that this simple does not occur in the induction. Here the schematic indicates that none of the $\SSS$-modules in the induction have this label, or any lower label in the rank order
--- the lowest possible rank is $l+2$. 
But then the upper-triangular simple decomposition property 
(see (\ref{pa:P2head}) and e.g. \cite[9.4]{amm2019tonal}) 
shows that the simple  $ \SSS_{n+2}(\mu,\mu') $ does not appear
as a simple factor,
and we are done by Schur's Lemma. }

\ppmm{ 
Next, continuing to parallel (\ref{pa:indicator})(IV), we 
suppose for a contradiction that some $\SSS_n(\nu,\nu')$ of rank $l+2$ (the next larger rank after rank $l$ here) is not simple,
and hence has a map in, necessarily from some 
$\SSS_{n}(\lambda,\lambda')$ of larger rank (by the upper-triangular property (\ref{pa:P2head})) as above. 
Observe from the restriction rules that any such   $\SSS_n(\nu,\nu')$
lies in the double-restriction of some $ \SSS_{n+2}(\mu,\mu') $
as above.
The other factors in 
the corresponding 
$\Res^{n+2}_{n} (\SSS_{n+2}(\mu,\mu')) $ are either in their own singleton block (since simple by assumption and not mapping into any lower rank module, since these are also all simple by assumption), 
or also of rank $l+2$, and hence neither extending nor extended by  $\SSS_n(\nu,\nu')$.}
Thus 
the factor $\SSS_n(\nu,\nu') $  in 
$\Res^{n+2}_{n} (\SSS_{n+2}(\mu,\mu')) $
can be made the first such filtration factor, so that 
$\Hom(\SSS_n(\nu,\nu'), \Res^{n+2}_{n} (\SSS_{n+2}(\mu,\mu'))) 
\neq 0$. 
This contradicts 
our Frobenius reciprocity identity. 
We must conclude that all modules of rank $l+2$ are simple. This completes the inductive step. 
}

The only-if direction is elementary, since the spine module $P^2_n E_{{0/1}}$ 
(case (a) or (b) respectively)
is indecomposable.
\end{proof}

\ignore{{ 
In \cite{amm2019tonal} it is shown that the tower of 2-tonal partitions algebras 
$0\subset P^2_1(\delta)\subset P^2_2(\delta)\subset P^2_3(\delta)\subset P^2_4(\delta)\subset\cdots$ 
satisfy the six {tower} axioms of \cite{cox2006representation},
\ppm{with 
$\cell_\nk(\ca{\emptyset},\emptyset)$ 
$[CORRECT??]$ 
in the role of the signifier module}. 
\ppm{[-but I am confused here, because \cite{cox2006representation} does not seem to talk about the signifier module!]}
\\
 Then by \cite[Theorem 1.1.]{cox2006representation} $P^2_\nk(\delta)$ is semisimple for all $\nk$ if and only if
$\cell_\nk(\emptyset/(1),\emptyset)$  is simple for all $\nk$.
\ppm{(as defined in (\ref{}))}
is simple for all $\nk$.
\ppm{[-I don't understand this yet!  Martin-Saleur proved the partition algebra case a long time before the paper \cite{cox2006representation} existed. And I think we need to explain the connection to Th.1.1 in \cite{cox2006representation} if we are going to claim to use it here. It is not obvious to me from what we write so far.]}
}}
\ppm{
}
\\


\ignore{{ 
{\it{Remark}}:   \ppm{[DELETE this now?]}
\ppm{In particular, satisfaction of}
The quasi-heredity axiom shows 
\ppm{[implies?]}
that the morphisms 
\ppm{[-what morphisms?]}
are only in one direction 
 determined by the quasi-heredity order on the indexing set of the simple modules. The Frobenius reciprocity and restriction rules axioms show that if there are any morphisms between two $P^2_n(\delta)$-modules labeled by $(\lambda,\mu)$ and $(\lambda',\mu')$ then the morphisms must appear in lower levels if  $(\lambda,\mu),(\lambda',\mu')\in\Lambda(P^2_{n-2}(\delta))$, see Figure~\ref{brat}. 
%

}}
\ignore{{

\subsection{On the hyperoctahedral group} 
\ppm{[DELETE this section now?]}

Let $\Z_2 \wr
\Sym_\kn$ be the wreath product of the $\Z_2$ with $\Sym_\kn$. The group $\Z_2 \wr
\Sym_\kn$ is sometimes known as the hyperoctahedral group. 
It can be identified with 
{the group of}
$\kn\times \kn$ signed permutation matrices. 

\mdef \label{pr:repZwrS}
The ordinary irreducible representations of $\Z_2 \wr
\Sym_\kn$ may be parameterised by the set $$
\Lambda(\Z_2 \wr
\Sym_\kn)=\{(\lambda, \mu)\mid \lambda\vdash s,\, \mu\vdash t \text{ and } s+t=\kn\}.
$$
\ppm{We do not give any details here, since none are needed, except that...}
The trivial representation is labeled by the pair $\left((\kn),\emptyset\right).$
For more on the representation theory of $\Z_2 \wr
\Sym_\kn$, see for example \cite{al1981representations, morris1981representations,  sniady2006gaussian, banjo}.

}}
\ignore{{

\subsection{Dimension of the head of the $P^2_n$ `spine' module
$\spine = P^2_n E_0$ for $\delta\in\N$ }   $\;$ 

\ppm{[{\bf DELETE} everything in this section except the Lemma at the end? That Lemma is what is needed, and it is not proved by anything here! (It is proved in \S\ref{ss:review}/\ref{ss:tonal-prep}.) Maybe somehow keep the Benkhart--Moon refs - even though we perhaps do not really use them anymore!]}

\newcommand{\V}{\mathsf{V}}
\newcommand{\G}{\mathsf{G}}   
\newcommand{\head}{{\mathsf{head}}}

\ppm{[it is maybe a bit hard to follow what is going on here. how about starting by explaining in terms of what happens in \cite{MartinSaleur94b}??]}
In this section, we follow \cite{benkart2016schur, benkart2017walks}  in recalling the Schur-Weyl duality between the 
\ppm{complex} 
representation theories of 
\ppm{a finite group}
$\G$ and 
\ppm{the centraliser algebras of its actions on $n$-th tensor powers of a fixed $\C \mathsf{G}$-module $V$,
denoted}
$\mathsf{Z}_\nk(\mathsf{G})$. 
This approach will allow us to use equation 28 of \cite{benkart2017walks} to determine the dimension of the head of the $P^2_\nk(\kn)$-module 
$\spine$.

\mdef 
Let $\G$ be finite group and $\mathsf{V}$ be a $\C \mathsf{G} $-module. 
Thus the action of $\C \G$ 
on $\V$ is given by a subalgebra of $\End(\V)$.
The centraliser algebra of the action of $\mathsf{G}$ on $\mathsf{V}^{\otimes \nk}$ is defined to be 
\ppm{[-are you defining, or introducing notation? I think we can define centraliser algebra for any action on any space. Then here maybe you could say the Z notation in these terms.]}\ca{[The reason why I did it in this way is I wanted to stay as close as possible to the notations and results in \cite{benkart2016schur, benkart2017walks}]}
\ppm{[-sure. But two wrongs dont make a right. Let us at least explain? It is very confusing to imply there is something special about this case when there is not.]}
\[
\mathsf{Z}_\nk(\mathsf{G})
\;=\;
\{z\in \mathsf{End}(\mathsf{V}^{\otimes \nk})\mid z(g.w)
=g.z(w)\text{ for all }g\in \mathsf{G}, w\in \mathsf{V}^{\otimes \nk} \}.
\]
\ppm{[-this depends on $\V$ but the notation suggests not. maybe add a $\V$? Or say that $\V$ is fixed?]}

\mdef  \label{de:Zlam}
Let $\Lambda(\G)$ 
be an indexing set for the simple $\mathsf{G}$-modules over $\C$.
Then let    $\Lambda_\nk(\mathsf{G})\subseteq\Lambda(\mathsf{G})$ 
be the indexing set of simple $\mathsf{G}$-modules that appear with a non-zero multiplicity in $\mathsf{V}^{\otimes \nk}$. 
The centraliser algebra $\mathsf{Z}_\nk(\mathsf{G})$ is a semisimple $\C$-algebra. 
 By the Schur-Weyl duality,
\begin{enumerate}
    \item[i)] there is a bijection between $\Lambda_\nk(\mathsf{G})$ and an indexing set of
    \ppm{simple}
 $\mathsf{Z}_\nk(\mathsf{G})$-modules. 
    Let $\{\mathsf{Z}^\lambda_\nk\mid \lambda\in\Lambda_\nk(\mathsf{G})\}$ 
    be a 
    corresponding 
    complete set of simple $\mathsf{Z}_\nk(\mathsf{G})$-modules up to isomorphism.
    \ppm{[-can you find a more mathematical way of saying this? mention Schur's Lemma maybe?]}
    \item[ii)]  for each $\lambda\in\Lambda_\nk(\mathsf{G})$ we have \[\dim\mathsf{Z}^\lambda_\nk=\text{multiplicity of the corresponding simple }\mathsf{G}\text{-module } \mathsf{G}_\lambda\text{ in }\mathsf{V}^{\otimes \nk}. \]
\end{enumerate}

\mdef 
In particular, if $\mathsf{G}=\Z_2 \wr
\Sym_\kn$, then $\mathsf{G}$ acts on $\mathsf{V}=\C^\kn$ as $\kn\times \kn$ signed permutation matrices relative to the standard basis of $\mathsf{V}$.
\\
For $\kn\geq 2\nk$, Tanabe in \cite{tanabe1997centralizer} proved the following isomorphism of $\C$-algebras  
\[
\mathsf{Z}_\nk(\Z_2 \wr
\Sym_\kn)\simeq P^2_\nk(\kn) 
\hspace{2cm}
{(\kn \geq 2\nk)}
\] 
(but of course this is false in general). 
For each positive integer $\kn$, the restriction of the ``Potts" representation of the partition algebra $P_\nk(\kn)$ on $\mathsf{V}^{\otimes \nk}$ ({see} \cite{martin2000partition}) gives,
\[
P^2_\nk(\kn)\twoheadrightarrow \mathsf{Z}_\nk(\Z_2 \wr
\Sym_\kn).
\]

\mdef 
Assume $\nk$ is even. The space of 
$\Z_2 \wr\Sym_\kn$-invariants 
in $\mathsf{V}^{\otimes \nk}$ is denoted by $\left(\mathsf{V}^{\otimes \nk}\right)^{\Z_2 \wr\Sym_\kn}$. 
Obviously this is the isotypic subspace where $\Z_2 \wr\Sym_\kn$ acts as on copies of the trivial module. 
Thus its dimension is the 
multiplicity of the trivial module, and hence the 
dimension of the corresponding module for the centraliser. That is 
\[
\dim\mathsf{Z}_{\nk}^{((\kn),\emptyset)}
=\dim ( (\mathsf{V}^{\otimes \nk})^{\Z_2 \wr\Sym_\kn})
\]

\mdef 
\ppm{PROPOSITION.}
Assume $\nk$ is even. 
Consider the index
$(\emptyset,\emptyset)\in\Lambda(P^2_\nk(\kn))$. 
We have 
\ppm{[better to put $\SSS_n(...)$ below 
- also better to have $\delta=k$ in notation?
- where is `head' introduced??
- change macro!? -]}
$$
\head(\spine)
=  \mathsf{Z}_{\nk}^{((\kn),\emptyset)}
$$ 
\ppm{[surely this is $\cong$ not $=$?]}
(recall the notation from (\ref{pr:repZwrS}))  
\ppm{[do we also need (\ref{de:Zlam})(i)? in that case I think the ``a bijection'' there needs to be specified?!]}
\ppm{[say something about the PROOF???]}

\mdef 
Furthermore
\[
\dim\mathsf{Z}_{\nk}^{((\kn),\emptyset)}
=\dim (\mathsf{V}^{\otimes \nk})^{\Z_2 \wr\Sym_n}
=\dim\mathsf{Z}_{\frac{\nk}{2}}(\Z_2 \wr\Sym_\kn))
=\sum_{t=1}^{\kn} T(\frac{\nk}{2}, t)
\]
\ppm{[-hard to believe it does not depend on $n$!!!]}\ca{[you are right, fixed it! this happened when I switched the k's and n']}
The second and third equality  follows from equation 28 of \cite{benkart2017walks}


Hence we have the following Lemma
\begin{lem}\label{dimL}
Let $\delta =\kn \ge 1$ then
$$\dim head(\cell_\nk(\emptyset, \emptyset)) 
= 
\sum_{t=1}^{\kn} T(\frac{\nk}{2}, t)
$$\qed
\end{lem}
NB: If $\delta =0$ then 
\ppm{formally??}
$\dim L(\emptyset, \emptyset) =0$. 
($P^2_\kn(\delta)$ is
cellular but not quasi-hereditary in this case. $L$ is defined to be
the cell module over its radical.
\ppm{[-this obviously does not work as written. but easy to fix.]})

}}

\subsection{Gram matrices for $P^2_n$  and the main Theorem} $\;$ 
\label{ss:avengers-end-game}




\ignore{{
\noindent 
\begin{figure}
\[
\includegraphics[width=5.5in]{bratf.eps}
\]
\caption{{Enhanced Bratelli diagram for $P^2_\nk(\delta)$ up to $\nk=4$. At each level the arrows represent the maps existing  between the standard modules for a specified $\delta$. For example, on level $\nk=4$ when $\delta=1$ there is a map from ($\emptyset$, (2)) to ($\emptyset$,$\emptyset$) with image of dimension $3$.\ca{[we do not need this fig anymore]} \label{brat}}}
\end{figure} 
}}

\mdef  
We will only need the $e_{(0,0)} = E_0$ case here, but 
there is a contravariant form 
$\langle-,-
\rangle_{e_{(\lambda,\mu)}}$ 
on each standard $P^2_{\nk}(\delta)$-module $\mathcal{S}_\nk(\lambda,\mu) \;$ 
(see Lemma 9.1 of \cite{amm2019tonal}), so that the conditions of (\ref{pa:gram00}) hold.
\ignore{{
defined by 
\begin{equation}\label{form}
    (xa^{\vec{m}} e_{(\lambda,\mu)})^{op}(ya^{\vec{m}} e_{(\lambda,\mu)})=\langle 
x,y\rangle_{e_{(\lambda,\mu)}}a^{\vec{m}} e_{(\lambda,\mu)}.
\end{equation} 
Here   ${(-)}^{op}$ is an 
}}
The
algebra antiautomorphism of $P^2_\nk(\delta)$ is again defined on the diagram basis by $p\mapsto p^*$ where $p^*$ is the diagram obtained from $p$ by a horizontal ``flip", as in (\ref{de:flip}). 

\mdef 
For example, consider $\nk=4$. 
Then Figure~\ref{gram} shows the Gram matrix of 
$\langle-,-\rangle_{e_{(\emptyset,\emptyset)}}$ 
on $\spine = P_n^2 \E_0$ 
- with respect to the basis and order shown in the figure. 
\begin{figure}
\includegraphics[width=2.5in]{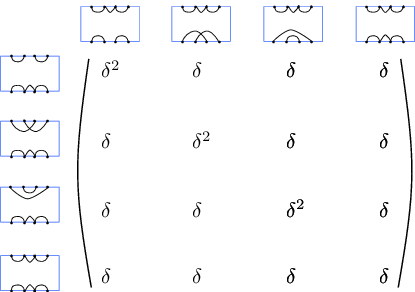}
\hspace{1.5cm}
\includegraphics[width=2.5in]{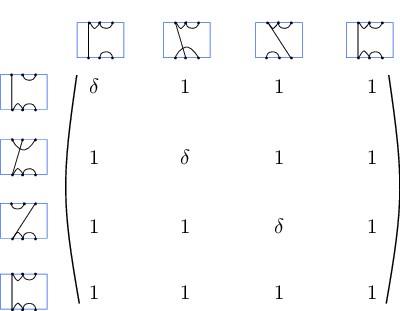}
\caption{{The first matrix is the Gram matrix of $\langle-,-
\rangle_{e_{(\emptyset,\emptyset)}}$ when $n=4$,
with respect to the indicated basis;  $\;$ 
and the second is the gram matrix of $\langle-,-
\rangle_{e_{((1),\emptyset)}}$ when $n=3$.} \label{gram}}
\end{figure}
It has determinant $\delta^4(\delta-1)^3$. 
Compare this with Fig.\ref{fig:bratvtwo}.
%
%
When $\delta=1$ 
we see from Fig.\ref{fig:bratvtwo} (the analogue here of Fig.\ref{fig:PnBratelli}) that 
there is a homomorphism 
from $\cell_4(0,(2))$         
on to the radical $rad(\cell_4(\emptyset,\emptyset))$ and 
$\dim rad(\cell_4(\emptyset,\emptyset))=3.$
$\;$ 
When $\delta=0$, 
there is a {\it surjective} map from $\cell_4((2),0)$ to $\cell_4(\emptyset,\emptyset)$,
so $\cell_4(\emptyset,\emptyset)$ is simple. 
But  $\E_0$ is nilpotent so $P^2_n \E_0$  is not the quotient of a projective,
just as in the $P_n(0)$ case.

\mdef 
For $\lambda,\mu$ a suitable pair of partitions let 
$\Gram^2_n(\lambda,\mu) = \Gram(\cell_\nk(\lambda,\mu))$ be the Gram matrix of 
$\langle -, -\rangle_{e_{(\lambda,\mu)}}$ on $\cell_\nk(\lambda,\mu)$ with respect to a given ordered basis.

We restrict  
computational 
attention to the Gram determinant of  $\cell_\nk(\emptyset,\emptyset)$. 
However observe that  
\begin{equation}\label{odd-even}
\delta^{\dim(\cell_{\nk+1}(\emptyset,\emptyset))}    
 \;   \Gram(\cell_\nk((1),\emptyset))
    \;   = \;  \Gram(\cell_{\nk+1}(\emptyset,\emptyset))
\hspace{1.6cm} \mbox{($n$ odd)}
\end{equation}
 cf. (\ref{eg:gramSn0}).
An example of the above equality is given in Figure~\ref{gram}.

\medskip


Next we parallel the Gram determinant polynomial degree identity  
for the $P_n$ case
from (\ref{pa:gram-det-order}). 

\begin{lem} \label{lem:degg}
The polynomial degree  
of the Gram determinant of $\cell_\nk(\emptyset,\emptyset)$ 
{in $P^2_\nk$} ($n$ even)
is 
$$
\deg( \det( \Gamma^2_n(\emptyset,\emptyset))) \; = \; 
\sum_{t=1}^{\nk/2} t \T(\frac{\nk}{2}, t)
$$
\end{lem}
\begin{proof}
Fix an order on the diagram basis of $\cell_\nk(\emptyset,\emptyset)$.  
The entries of the Gram matrix are all monomials in $\delta$. 
Let $x$ and $y$ be two diagram basis 
\ppmm{elements}
of $\spine$. Then  
$\langle x, y\rangle_{e_{(\emptyset,\emptyset)}}
 = \delta^{\text{ number of parts of }x\setminus \{\underline{n}'\}} $ 
 when $x=y$, and by the pigeonhole principle 
$\;$ degree of $\langle x, x\rangle_{e_{(\emptyset,\emptyset)}}\geq$ degree of $ \langle x, y\rangle_{e_{(\emptyset,\emptyset)}}$ 
when $x\neq y$. 
That is, in each row/column the degree of the diagonal entry is never strictly dominated by the degree of the other entries. 
Hence the degree of 
 the determinant of Gram matrix 
 $\Gram(\spine)$   
 equals the 
 degree of the product of the diagonal entries of 
 $\Gram(\spine).$ 
 There are $\T(\frac{\nk}{2}, t)$ basis 
 {elements}
 $x$ such that the number of parts of $x\setminus \{\underline{\nk}'\}$ is $t$. 
 Hence the degree  
 is 
$
\sum_{t=1}^{\nk/2} t \T(\frac{\nk}{2}, t).
$
\end{proof}


Following (\ref{eq:gram1}) 
and (\ref{eq:gramfactor2})
define 
exponent 
$\beta_{\delta_c}$ for $\delta_c \in \C$ 
(and fixed $n$)
by setting
\beq  \label{eq:gram21} 
\det(\Gram^2_n(\oo)) = \prod_{\delta_c} (\delta-\delta_c)^{\beta_{\delta_c}} 
\eq 

\begin{lem}\label{gramfact}
For $k \in \Fk$
let 
$ \; L_\nk^{\delta=\kn}(\emptyset,\emptyset)
\; = \; 
head \;\cell_\nk(\emptyset,\emptyset) \;$
{in $P^2_\nk(\kn)$}. 
\begin{enumerate}
 \item[(I)] The identities in Lem.\ref{dimL} and Lem.\ref{lem:dimS} yield:
$$
\dim \cell_\nk(\emptyset,\emptyset)  + 
\sum_{\delta=\kn \in \N} 
\left( \dim \cell_\nk(\emptyset,\emptyset) 
- \dim
L_\nk^{\delta=\kn}(\emptyset,\emptyset)\right) 
\;\;\;  =  \;\;\;    {\sum_{t=1}^{\nk/2}}t \T(\frac{\nk}{2}, t).
$$
\item[(II)]
The Gram determinant of $\cell_\nk(\emptyset,\emptyset)$ 
{in $P^2_\nk$} ($n$ even)
factorises as 
\beq \label{eq:detfact2}
\det
(\Gram(\cell_\nk(\emptyset,\emptyset) ))
\; = \; 
\delta^{\dim \cell_\nk(\emptyset,\emptyset)}
\prod_{k=1}^{\nk/2-1}
(\delta-\kn)^{
\dim \cell_\nk(\emptyset,\emptyset) - \dim L_\nk^{\delta=\kn}(\emptyset,\emptyset)
}
\eq 
\end{enumerate}
\end{lem}
\begin{proof}
(I) The argument is directly parallel to that for the ordinary case in (\ref{pa:simpleS0}). 
\\
(II) 
First note that, for any two diagram basis  {elements}
$x$ and $y$ of $\cell_\nk(\emptyset,\emptyset)) $ the expression 
$ \langle x, y\rangle_{e_{(\emptyset,\emptyset)}}$ 
is a monomial in $\delta$ of a positive degree. Hence 
$\delta^{\dim \cell_\nk(\emptyset,\emptyset)}$ is a factor of 
$\det (\Gram(\cell_\nk(\emptyset,\emptyset) ))$. 
From (\ref{eq:alphabound2}) 
in Prop.\ref{pr:hid22}
we have that the expression on the right of (\ref{eq:detfact2}) 
must be a factor of the determinant. 
But from 
(\ref{lem:degg}) 
and (I)   
we see that this saturates the bound on degree, giving the identity.
\ignore{{ 
\footnote{
\ppm{[now DELETE the rest???]}\ca{[yes]}
For each $\delta=\kn\geq1$, by Proposition~\ref{dimL=rank} we have 
$\;\dim L_\nk^{\delta=\kn}(\emptyset,\emptyset)
\; =\;  rank( \langle -, -\rangle_{e_{(\emptyset,\emptyset)}}|_{\delta=k})$.
But 
$$
\; rank \langle -, -\rangle_{e_{(\emptyset,\emptyset)}}
=\dim \cell_\nk(\emptyset,\emptyset)-\dim
rad(\cell_\nk(\emptyset,\emptyset)),
$$
\ppm{by ????}.
Hence 
\ppm{[-explain?! - do we need to say that $\C[\delta]$ is a PID so gram matrix has a Smith form with entries in $\C[\delta]$? But evaluating at $k$ the rank is determined by the number of diagonal entries in the Smith form that vanish...]}
$\; (\delta-\kn)^{\dim
rad(\cell_\nk(\emptyset,\emptyset))}$ is also a factor of 
$\det
(\Gram(\cell_\nk(\emptyset,\emptyset) ))$. 
Hence, the degree of $\det
(\Gram(\cell_\nk(\emptyset,\emptyset) ))$ is bounded below by $\dim \cell_\nk(\emptyset,\emptyset)+\sum \dim
rad(\cell_\nk(\emptyset,\emptyset))$, and comparing these two expressions one can see that the bound in fact saturates. } 
}}
\end{proof}



\begin{eg}
\ignore{{ 
\item The 
\ppm{standard-basis [??]}
Gram matrix of the $P^1_2(\delta)$ module $\mathcal{S}(\emptyset)$ is 
$\begin{pmatrix}
\delta^2 & \delta\\
\delta & \delta\\
\end{pmatrix}$ 
and has determinant $\delta^2(\delta-1)$.
\item 
The Gram matrix of the $P^1_3(\delta)$ module $\mathcal{S}(\emptyset)$ is 
$\begin{pmatrix}
\delta^3 & \delta^2 & \delta^2 & \delta^2 & \delta\\
\delta^2 & \delta^2 & \delta & \delta & \delta\\
\delta^2 & \delta & \delta^2 & \delta & \delta\\
\delta^2 & \delta & \delta & \delta^2 & \delta\\
\delta & \delta & \delta & \delta & \delta\\
\end{pmatrix}$ and has determinant $\delta^5(\delta-1)^4(\delta-2)$.
    \item The Gram matrix of the $P^1_4(\delta)$ module $\mathcal{S}(\emptyset)$ is \[
\begin{pmatrix}
\delta^4 & \delta^3 & \delta^3 & \delta^3 & \delta^3 & \delta^3 & \delta^3 & \delta^2 & \delta^2 & \delta^2 & \delta^2 & \delta^2 & \delta^2 & \delta^2 & \delta \\
\delta^3 & \delta^3 & \delta^2 & \delta^2 & \delta^2 & \delta^2 & \delta^2 & \delta^2 & \delta & \delta & \delta^2 & \delta^2 & \delta & \delta & \delta \\
\delta^3 & \delta^2 & \delta^3 & \delta^2 & \delta^2 & \delta^2 & \delta^2 & \delta & \delta^2 & \delta & \delta^2 & \delta & \delta^2 & \delta & \delta \\
\delta^3 & \delta^2 & \delta^2 & \delta^3 & \delta^2 & \delta^2 & \delta^2 & \delta & \delta & \delta^2 & \delta & \delta^2 & \delta^2 & \delta & \delta \\
\delta^3 & \delta^2 & \delta^2 & \delta^2 & \delta^3 & \delta^2 & \delta^2 & \delta & \delta & \delta^2 & \delta^2 & \delta & \delta & \delta^2 & \delta \\
\delta^3 & \delta^2 & \delta^2 & \delta^2 & \delta^2 & \delta^3 & \delta^2 & \delta & \delta^2 & \delta & \delta & \delta^2 & \delta & \delta^2 & \delta \\
\delta^3 & \delta^2 & \delta^2 & \delta^2 & \delta^2 & \delta^2 & \delta^3 & \delta^2 & \delta & \delta & \delta & \delta & \delta^2 & \delta^2 & \delta \\
\delta^2 & \delta^2 & \delta & \delta & \delta & \delta & \delta^2 & \delta^2 & \delta & \delta & \delta & \delta & \delta & \delta & \delta \\
\delta^2 & \delta & \delta^2 & \delta & \delta & \delta^2 & \delta & \delta & \delta^2 & \delta & \delta & \delta & \delta & \delta & \delta \\
\delta^2 & \delta & \delta & \delta^2 & \delta^2 & \delta & \delta & \delta & \delta & \delta^2 & \delta & \delta & \delta & \delta & \delta \\
\delta^2 & \delta^2 & \delta^2 & \delta & \delta^2 & \delta & \delta & \delta & \delta & \delta & \delta^2 & \delta & \delta & \delta & \delta \\
\delta^2 & \delta^2 & \delta & \delta^2 & \delta & \delta^2 & \delta & \delta & \delta & \delta & \delta & \delta^2 & \delta & \delta & \delta \\
\delta^2 & \delta & \delta^2 & \delta^2 & \delta & \delta & \delta^2 & \delta & \delta & \delta & \delta & \delta & \delta^2 & \delta & \delta \\
\delta^2 & \delta & \delta & \delta & \delta^2 & \delta^2 & \delta^2 & \delta & \delta & \delta & \delta & \delta & \delta & \delta^2 & \delta \\
\delta & \delta & \delta & \delta & \delta & \delta & \delta & \delta & \delta & \delta & \delta & \delta & \delta & \delta & \delta \\
\end{pmatrix}
\] and has determinant $\delta^{15}(\delta-1)^{{14}}(\delta-2)^7(\delta-3)$.
}}  
The Gram matrix of the $P^2_6(\delta)$ module $\cell_6(\emptyset,\emptyset)$ is a $31\times 31$ matrix and its determinant is a polynomial of degree $3\cdot15+2\cdot15+1=76$. 
Taking into account the known simple modules 
we thus have
{
$$
\det(
\Gram 
(\cell_6(\emptyset,\emptyset))) \; = \; \delta^{31}(\delta-1)^{30=31-1}(\delta-2)^{15=31-(15+1)}
$$
}

\end{eg}

\begin{prop}\label{simpS}
Fixing the ground field to be $\C$, \ppmm{and fixing $\delta \in \C$,}
\begin{enumerate}
    \item[(I)] if $n$ is even then the $P^2_{\nk}(\delta)$-module $\cell_{\nk}(\emptyset,\emptyset)$ is simple if and only if  $\delta\not\in\{ 
1,2,\dots,\nk/2 -1\}$.
\item[(II)] if $n$ is odd then the $P^2_{\nk}(\delta)$-module $\cell_{\nk}({(1)},\emptyset)$ is simple if and only if  $\delta\not\in\{ 
1,2,\dots,(\nk-1)/2\}$.
\end{enumerate}

    \end{prop}
\begin{proof}
First note that $\cell_\nk(\emptyset/{(1)},\emptyset)$ satisfies the conditions of 
   Prop.\ref{pr:hid22}, unless $\delta=0$ and $n$ even. 
For part (I), using Lemma~\ref{gramfact} 
(putting aside $\delta = 0$ which has a non-singular rescaling, see \ref{pa:delta=0})
the form is singular if and only if   $\delta\in\{1,2,\dots,\nk/2 -1\}$.
\ppmm{For the $\delta = 0$ case the Gram matrix has, as noted 
((\ref{odd-even}) and cf. (\ref{pa:delta=0}) {\em{et seq}}), 
the non-singular rescaling, so the module is isomorphic to its contravariant dual, however since the generating element for the module is not (primitive) idempotent we need to eliminate the possibility of a non-simple such module (for example a module with Loewy structure $a/b/a$). For this we can use the $G$-functor associated to $\oWWW$ (which is idempotent for all $\delta$). The caveat here is 
that this requires $n>2$ and hence $n \geq 4$, however the low-rank case can be dealt with by explicit calculation  
(on the module $\cell_{1}({(1)},\emptyset)$, 
resp. $\cell_{2}(\emptyset,\emptyset)$).}
\\ 
(II) Using Eq.\ref{odd-even} and Lemma~\ref{gramfact}, the form 
$\langle-,-\rangle_{e_{((1),\emptyset)}}$ is singular if and only if $\delta\in\{1,2,\dots,(\nk-1)/2\}$.
\end{proof}


\begin{thm}
\label{th:main}
Fix 
{the ground field to be $\C$, and }
let $\delta\in \mathbb{C}$. 
{
The 2-tonal partition algebras $P^2_\nk(\delta)$ are semisimple for all $\nk \in \N$
if and only if  
$\delta\not\in\N_0$. 
\ppmm{(If $\delta=0$ then $P^2_\nk(\delta)$ is semisimple for $\nk>0$ if and only if  $\nk$ is odd.)}}
\end{thm}
\begin{proof}
{By Lemma~\ref{frecip} the algebras $P^2_\nk(\delta)$ 
with $\delta\neq 0$ 
are semisimple for all $\nk \in \N$ if and only if 
the modules}
{$\cell_\nk({\emptyset}/{(1)},\emptyset)$ are simple for all $\nk \in \N$. }
The case of $\delta=0$ is immediate (observe that $E_0$ lies in the radical),
as is the only-if part 
(compare Lemma~\ref{dimL} with \ref{lem:dimS})
so, 
finally, the result follows from Proposition \ref{simpS}. 
\end{proof}

\medskip 

\medskip 

\appendix 

\ignore{{ 
\section{The ordinary partition algebra}  


\mdef  
For a  positive integer $n$, set $\underline{n}:=\{1,2,\dots,n\}$, $\underline{n}':=\{1',2',\dots,n'\}$ and $\underline{n}'':=\{1'',2'',\dots,n''\}$. 
\blue{Let $P_S$ be the set of all set partitions of a set $S$. }
\ppm{[-isn't it better to use the same font as before?]}

\mdef 
Define 
\begin{align*} 
\sigma^+:P_{\underline{n}\cup\underline{n}'}&\to P_{\underline{n}\cup\underline{n}''} \\ 
p&\mapsto \sigma^{+}(p)
\end{align*}
 where $\sigma^{+}(p)$ is a partition of $\underline{n}\cup\underline{n}''$ obtained from $p$ by replacing each $i'$ by $i''$ and fixing all the elements of $\underline{n}.$ 
For example, $\sigma^+(\{\{1,2'\}, \{2\},\{1'\}\})=\{\{1,2''\}, \{2\},\{1''\}\}$. Analogously, define 
 \begin{align*} 
\sigma^-:P_{\underline{n}\cup\underline{n}'}&\to P_{\underline{n''}\cup\underline{n}'} \\ 
p&\mapsto \sigma^{-}(p)
\end{align*}
 where $\sigma^{-}(p)$ is a partition of $\underline{n''}\cup\underline{n}'$ obtained from the partition $p$ by replacing each $i$ by $i''$ and fixing all the elements of $\underline{n}'.$ 
\\ Let $\overline{p}$ be the equivalence relation on a set $S$ obtained from a partition $p$ of $S$. Set $P_n:=P_{\underline{n}\cup\underline{n}'}$. For  $p,q\in P_n$ define $p\circ q$ to be the partition of ${\underline{n}\cup\underline{n}'}$ corresponding to the equivalence relation $(\overline{\sigma^+(p)}\cup_{cl}\overline{\sigma^-(q)})\downharpoonleft_{{\underline{n}\cup\underline{n}'}}$, where $\overline{\sigma^+(p)}\cup_{cl}\overline{\sigma^-(q)}$ is the smallest equivalence relation which contains both $\overline{\sigma^+(p)}$ and $\overline{\sigma^-(q)}$, and $\downharpoonleft_{{\underline{n}\cup\underline{n}'}}$ is restriction of the relation to the set ${{\underline{n}\cup\underline{n}'}}.$ 

Let $K$ be a commutative ring and $\delta\in K$. Then the partition algebra $P_\nk(\delta)$ is a free $K$-module with basis $P_n$ and the multiplication of basis elements is given by $$p.q=\delta^{ \alpha{(p,q)} }p\circ q$$
where $\alpha{(p,q)}=$ number of parts removed in the restriction step, $\downharpoonleft_{{\underline{n}\cup\underline{n}'}}$. 
\blue{[not finished yet CA...]}

}}





\bibliographystyle{amsplain}
\bibliography{bib/local}

\end{document}